\documentclass[12pt]{amsart}

 \newtheorem{thm}{Theorem}[section]
 
 \newtheorem{lem}[thm]{Lemma}
 
 \theoremstyle{definition}
 
 \theoremstyle{remark}
 
 \numberwithin{equation}{section}

\newcommand{\Real}{\mathbb{R}}
\newcommand{\Complex}{\mathbb{C}}

\newcommand{\LD}{\mathcal{L}}

\newcommand{\tr}{\textbf{tr}}
\newcommand{\hess}{\textbf{Hess}}

\newcommand{\E}{\mathcal E}

\newcommand{\cF}{\mathcal F}

 \newcommand{\bJ}{\textbf J}
\newcommand{\ver}{\mathcal V}
\newcommand{\hor}{\mathcal H}
\newcommand{\vs}{\mathfrak{V}}
\newcommand{\D}{\mathcal D}

\newcommand{\spn}{\text{span}}
\newcommand{\Rm}{\text{Rm}}
\begin{document}

\title[Comparison theorems on Sasakian manifolds]{Bishop and Laplacian comparison theorems on Sasakian manifolds}
\author{Paul W. Y. Lee}
\address{Room 216, Lady Shaw Building, The Chinese University of Hong
Kong, Shatin, Hong Kong}
\email{wylee@math.cuhk.edu.hk}
\author{Chengbo Li}
\address{Department of
Mathematics, Tianjin University, Tianjin, 300072, P.R.China}
\email{chengboli@tju.edu.cn}
\begin{abstract}
We prove a Bishop volume comparison theorem and a Laplacian comparison theorem for a natural sub-Riemannian structure defined on Sasakian manifolds. This generalizes the earlier work in \cite{Hu, ChYa, AgLe2} for the three dimensional case.
\end{abstract}

\thanks{The first author's research was supported by the Research Grant Council of Hong Kong (RGC Ref. No. CUHK404512). The second author was supported in part by the National Natural Science Foundation of China (Grant No. 11201330)}.

\date{\today}

\maketitle

\section{Introduction}

Bishop volume comparison theorem and Laplacian comparison theorem are basic tools in Riemannian geometry and geometric analysis. In this paper, we prove an analogue for a natural sub-Riemannian structure defined on a Sasakian manifold.

Recall that a Sasakian manifold is a $2n+1$-dimensional manifold $M$ equipped with the an almost contact structure $(\bJ,\alpha_0,v_0)$ and a Riemannian metric $\left<\cdot,\cdot\right>$ satisfying certain compatibility conditions (see Section \ref{Sasakian} for the definitions). The restriction of the Riemannian metric on the distribution $\D:=\ker\alpha_0$ defines a sub-Riemannian structure. Let $B_x(R)$ be the sub-Riemannian ball of radius $R$ centered at $x$ and let $\eta$ be the Riemannian volume form of the Riemannian metric $\left<\cdot,\cdot\right>$. The Heisenberg group and the complex Hopf fibration are well-known Sasakian manifolds (see Section \ref{Model} for more detail). Their volume forms are denoted, respectively, by $\eta_0$ and $\eta_H$. We also denote their sub-Riemannian balls by  and $B_0(R)$ and $B_H(R)$, respectively. The following Bishop type volume comparison theorems generalize the earlier three dimensional case in \cite{Hu, ChYa, AgLe2}.

\begin{thm}\label{main1-1}
Assume that the Tanaka-Webster curvature $\Rm^*$ of the Sasakian manifold satisfies
\begin{enumerate}
\item $\left<\Rm^*(\bJ v,v)v,\bJ v\right>\geq 0$,
\item $\sum_{i=1}^{2n-2}\left<\Rm^*(w_i,v)v,w_i\right>\geq 0$,
\end{enumerate}
where $v$ is any vector in $\D$ and $w_1,...,w_{2n-2}$ in an orthonormal frame of $\{v_0,v,\bJ v\}^\perp$. Then
\[
\eta(B_x(R))\leq\eta_0(B_0(R)).
\]
Moreover, equality holds only if
\begin{enumerate}
\item $\left<\Rm^*(\bJ v,v)v,\bJ v\right>= 0$,
\item $\sum_{i=1}^{2n-2}\left<\Rm^*(w_i,v)v,w_i\right>= 0$,
\end{enumerate}
on $B_x(R)$.
\end{thm}

\begin{thm}\label{main1-2}
Assume that the Tanaka-Webster curvature $\Rm^*$ of the Sasakian manifold satisfies
\begin{enumerate}
\item $\left<\Rm^*(\bJ v,v)v,\bJ v\right>\geq 4|v|^4$,
\item $\sum_{i=1}^{2n-2}\left<\Rm^*(w_i,v)v,w_i\right>\geq (2n-2)|v|^2$,
\end{enumerate}
where $v$ is any vector in $\D$ and $w_1,...,w_{2n-2}$ in an orthonormal frame of $\{v_0,v,\bJ v\}^\perp$.  Then
\[
\eta(B_x(R))\leq\eta_H(B_H(R)).
\]
Moreover, equality holds only if
\begin{enumerate}
\item $\left<\Rm^*(\bJ v,v)v,\bJ v\right>= 4|v|^4$,
\item $\sum_{i=1}^{2n-2}\left<\Rm^*(w_i,v)v,w_i\right> =(2n-2)|v|^2$,
\end{enumerate}
on $B_x(R)$.
\end{thm}

A Laplacian type comparison theorem generalizing the one in \cite{AgLe2} also holds. Recall that sub-Laplacian $\Delta_H$ is defined by
\[
\Delta f=\sum_{i=1}^{2n}\left<\nabla_{v_i}\nabla f,v_i\right>,
\]
where $v_1,...,v_{2n}$ is an orthonormal frame in $\D$.

\begin{thm}\label{main2}
Let $x_0$ be a point in $M$ and let $d(x):=d(x_0,x)$ be the sub-Riemannian distance from the point $x_0$. Assume that the Tanaka-Webster curvature $\Rm^*$ of the Sasakian manifold satisfies
\begin{enumerate}
\item $\left<\Rm^*(\bJ v,v)v,\bJ v\right>\geq k_1|v|^4$,
\item $\sum_{i=1}^{2n-2}\left<\Rm^*(w_i,v)v,w_i\right>\geq (2n-2)k_2|v|^2$,
\end{enumerate}
for some constants $k_1$ and $k_2$, where $v$ is any vector in $\D$ and $w_1,...,w_{2n-2}$ in an orthonormal frame of $\{v_0,v,\bJ v\}^\perp$.  Then
\[
\Delta_H d\leq h(d,v_0(d)),
\]
where $\mathfrak k_1(r,z)=z^2+k_1r^2$, $\mathfrak k_2(r,z)=\frac{1}{4}z^2+k_2r^2$, and
\[
\begin{split}
&h(r,z)=\frac{\sqrt{\mathfrak k_1}(\sin(\sqrt{\mathfrak k_1}-\sqrt{\mathfrak k_1}\cos(\sqrt{\mathfrak k_1}))}{r(2-2\cos(\sqrt{\mathfrak k_1})-\sqrt{\mathfrak k_1}\sin(\sqrt{\mathfrak k_1}))}+\frac{(2n-1)\sqrt{\mathfrak k_2}\cot(\sqrt{\mathfrak k_2})}{r}
\end{split}
\]
if $\mathfrak k_1\geq 0$ and $\mathfrak k_2\geq 0$,
\[
\begin{split}
&h(r,z)=\frac{\sqrt{\mathfrak k_1}(\sqrt{\mathfrak k_1}\cosh(\sqrt{\mathfrak k_1}))-\sinh(\sqrt{\mathfrak k_1}))}{r(2-2\cosh(\sqrt{-\mathfrak k_1})+\sqrt{-\mathfrak k_1}\sinh(\sqrt{-\mathfrak k_1}))}+\frac{(2n-1)\sqrt{\mathfrak k_2}\cot(\sqrt{\mathfrak k_2})}{r}
\end{split}
\]
if $\mathfrak k_1\geq 0$ and $\mathfrak k_2\leq 0$,
\[
\begin{split}
&h(r,z)=\frac{\sqrt{\mathfrak k_1}(\sin(\sqrt{\mathfrak k_1}-\sqrt{\mathfrak k_1}\cos(\sqrt{\mathfrak k_1}))}{r(2-2\cos(\sqrt{\mathfrak k_1})-\sqrt{\mathfrak k_1}\sin(\sqrt{\mathfrak k_1}))}+\frac{(2n-1)\sqrt{\mathfrak k_2}\coth(\sqrt{\mathfrak k_2})}{r}
\end{split}
\]
if $\mathfrak k_1\leq 0$ and $\mathfrak k_2\geq 0$,
\[
\begin{split}
&h(r,z)=\frac{\sqrt{\mathfrak k_1}(\sqrt{\mathfrak k_1}\cosh(\sqrt{\mathfrak k_1}))-\sinh(\sqrt{\mathfrak k_1}))}{r(2-2\cosh(\sqrt{-\mathfrak k_1})+\sqrt{-\mathfrak k_1}\sinh(\sqrt{-\mathfrak k_1}))}+\frac{(2n-1)\sqrt{\mathfrak k_2}\coth(\sqrt{\mathfrak k_2})}{r}
\end{split}
\]
if $\mathfrak k_1\leq 0$ and $\mathfrak k_2\leq 0$.
\end{thm}
A version of Hessian comparison theorem as in \cite{AgLe2} also hold. The proof is very similar to and simpler than that of Theorem \ref{main2}. We omit the statement since it is rather lengthy.

The paper is organized as follows. In section \ref{frame}, we recall the construction of the canonical frame introduced in \cite{LiZe}. In section \ref{Sasakian}, we recall the definition of Sasakian manifolds. We also recall the definition of parallel adapted frame introduced in \cite{Le1} which simplifies the computation of the canonical frame, which is done in section \ref{curvature}. In section \ref{conjestimate}, we prove a first conjugate time estimate under the lower bounds on the Tanaka-Webster curvature. In section \ref{Model}, we discuss the Heisenberg group, the complex Hopf fibration, and their sub-Riemannian cut locus. The volume estimate and the proof of Theorem \ref{main1-1} and \ref{main1-2} are done in section \ref{volestimate}. Finally, section \ref{Laplace} is devoted to the proof of Theorem \ref{main2}.

\smallskip

\section{Canonical frames and curvatures of a Jacobi curve}\label{frame}

In this section, we recall how to construct canonical frames and define the curvature of a curve in Lagrangian Grassmannian. We will only do the construction in our simplified setting. For the most general discussion, see \cite{LiZe}. For completeness, we will also include the full proof of the results in our case.

Let $t\mapsto J(t)$ be a curve in the Lagrangian Grassmannian of a symplectic vector space $\vs$. Let $g^0_t$ be the bilinear form on $J(t)$ defined by
\[
g_t^0(e,e)=\omega(\dot e(t),e),
\]
where $e(\cdot)$ is any curve in $J$ such that $e(t)=e$.

Assume that the curve $J$ is monotone which means that $g^0_t$ is non-negative definite for each $t$. Let $J^{-1}$, $J^1$, and $J^2$ be defined by
\[
\begin{split}
&J^{-2}(t)=\{e(t)|\dot e(t),\ddot e(t)\in J(t)\},\\
&J^{-1}(t)=\{e(t)|\dot e(t)\in J(t)\},\\
&J^1(t)=\textbf{span}\{e(t),\dot e(t)|e(\cdot)\in J\}=(J^{-1})^\angle\\
&J^2(t)=\textbf{span}\{e(t),\dot e(t),\ddot e(t)|e(\cdot)\in J\}=(J^{-2})^\angle
\end{split}
\]
where the superscript $W^\angle$ denotes the symplectic complement of the subspace $W$.

We will consider the case $J^1\neq \mathfrak V$ and $J^2=\mathfrak V$. Assume that $J$ and $J^{-1}$ have dimensions $N$ and $k$, respectively.

\begin{thm}\cite{LiZe}\label{structure}
Under the above assumptions, there exists a family of frames $E^1(t)=(E^1_1(t),...,E^1_k(t))^T$, $E^2(t)=(E^2_1(t),...,E^2_k(t))^T$, $E^3(t)=(E^3_1(t),...,E^3_{N-2k}(t))^T$, $F^1(t)=(F^1_1(t),...,F^1_k(t))^T$, $F^2(t)=(F^2_1(t),...,F^2_k(t))^T$, $F^3(t)=(F^3_1(t),...,F^3_{N-2k}(t))^T$ such that
\begin{enumerate}
\item $E(t)=(E^1(t),E^2(t),E^3(t))^T, F(t)=(F^1(t),F^2(t),F^3(t))^T$ is a symplectic basis for each $t$,
\item $E^1(t)$ is a basis of $J^{-1}(t)$,
\item $\dot E(t)=C_1E(t)+C_2F(t),\quad \dot F(t)=-R(t)E(t)-C_1^TF(t)$,
\end{enumerate}
where
\[
\begin{split}
&C_1=\left(
           \begin{array}{ccc}
             0 & I & 0 \\
             0 & 0 & 0 \\
             0 & 0 & 0 \\
           \end{array}
         \right),
C_2=\left(
           \begin{array}{ccc}
             0 & 0 & 0 \\
             0 & I & 0 \\
             0 & 0 & I \\
           \end{array}
         \right),\\
&R(t)=\left(
           \begin{array}{ccc}
             R^{11}(t) & 0 & R^{13}(t) \\
             0 & R^{22}(t) & R^{23}(t) \\
             R^{31}(t) & R^{32}(t) & R^{33}(t) \\
           \end{array}
         \right),
\end{split}
\]
and $R(t)$ is symmetric.
\end{thm}

The frame $(E^1,E^2,E^3,F^1,F^2,F^3)$ is called a canonical frame of the curve $J$ and the coefficients $R^{ij}$ are the curvatures of the curve $J$. We also write the above equations as
\begin{equation}\label{structural}
\begin{split}
&\dot E^1(t)=E^2(t),\quad \dot E^2(t)=F^2(t),\quad \dot E^3(t)=F^3(t), \\
&\dot F^1(t)=-R^{11}(t)E^1(t)-R^{13}(t)E^3(t),\\
&\dot F^2(t)=-R^{22}(t)E^2(t)-R^{23}(t)E^3(t)-F^1(t),\\
&\dot F^3(t)=-R^{31}(t)E^1(t)-R^{32}(t)E^2(t)-R^{33}(t)E^3(t).
\end{split}
\end{equation}

\begin{proof}
Let $g^1_t$ be the bilinear form on $J^{-1}(t)$ defined by
\[
g^1_t(e,e)=\omega(\ddot e(t),\dot e(t)),
\]
where $e(\cdot)$ is any curve in $J^{-1}$ such that $e(t)=e$.

The bilinear form $g^1_t$ is well-defined. Indeed, let $e_1(\cdot)$ and $e_2(\cdot)$ be two curves in $J^{-1}(\cdot)$ such that $e_1(t)=e_2(t)$. Let $e_3(\cdot)$ be a curve in $J^1$. Since $J^{-1}$ is the skew-orthogonal complement of $J^1$, we have
\[
\omega(e_1(s)-e_2(s),e_3(s))=0.
\]
By differentiating the above expression, we have
\[
\omega(\dot e_1(t)-\dot e_2(t),e_3(t))=0.
\]

Since $e_3(t)$ is arbitrary, we see that $\dot e_1(t)-\dot e_2(t)$ is contained in $J^{-1}(t)$. On the other hand, since $\dot e_1(s)$ and $\dot e_2(s)$ are contained in $J(s)$ and $J(s)$ is Lagrangian, we have
\[
\omega(\dot e_1(s)-\dot e_2(s), \dot e_1(s))=0.
\]
Since $\ddot e_1(s)$ and $\ddot e_2(s)$ are contained in $J^1(s)$, we have, by differentiating the above expression,
\[
\omega(\ddot e_1(t), \dot e_1(t))=\omega(\ddot e_2(t),\dot e_1(t))=\omega(\ddot e_2(t),\dot e_2(t))
\]
and $g^1_t$ is well-defined.

Next, we claim that $g^1_t$ is an inner product and there exists a family of basis $E^1(\cdot)=(E^1_1(\cdot),...,E^1_k(\cdot))^T$ along $J^{-1}(\cdot)$ which is orthonormal with respect to $g^1$ such that
\[
\omega(\ddot E^1,\ddot E^1)=0.
\]
Here if $E=(E_1,...,E_k)$ and $F=(F_1,...,F_k)$ are two vectors, then $\omega(E,F)$ denotes the matrix with $ij$-th entry equal to $\omega(E_i,F_j)$.

Moreover, the family $E^1(\cdot)$ is unique up to multiplication by an orthogonal matrix (independent of time $t$). Indeed let $\bar E(\cdot)$ be a family of basis in $J^{-1}(\cdot)$. Since $J^{-2}=(J^2)^\angle=\{0\}$, $\dot{\bar E}(t)$ is not in $J^{-1}(t)$ which is the kernel of $g^0_t$. Therefore,
\[
g^1_t(\bar E,\bar E)=g^0_t(\dot {\bar E},\dot{\bar E})
\]
is positive definite.

Let $\bar E^1=(\bar E^1_1,...,\bar E^1_k)^T$ be a family of curves in $J^{-1}$ such that
\[
(\bar E^1_1(t),...,\bar E^1_k(t))^T
\]
is an orthonormal basis of $J^{-1}$ with respect to $g^1_t$. Then any other such family is given by $E(t)=O(t)\bar E(t)$. Therefore,
\[
\begin{split}
\omega\left(\frac{d^2}{dt^2} (O E^1),\frac{d^2}{dt^2}(OE^1)\right)&=\omega\left(2\dot O \dot E^1+O\ddot E^1,2\dot O \dot E^1+O\ddot E^1\right)\\
&=-2\dot OO^T+2O\dot O^T+O\omega(\ddot E^1,\ddot E^1)O^T\\
&=-4\dot OO^T+O\omega(\ddot E^1,\ddot E^1)O^T.
\end{split}
\]
Here, the first equality holds since $E^1(t)$ is contained in $J^{-1}(t)$ and $\dot E^1(t), \ddot E^1(t)$ are contained in $J^1(t)$. The second equality holds since $\omega(\ddot E^1(t),\dot E^1(t))=g_t^1(E^1(t),E^1(t))=I$ and $\dot E^1(t)$ is in $J(t)$. The last equality holds since $O(t)$ is orthogonal.

It follows that $E^1$ satisfies $\omega(\ddot E^1,\ddot E^1)=0$ if and only if $O$ is a solution to the equation
\[
\dot O=\frac{1}{4}O\omega(\ddot E^1,\ddot E^1).
\]
This finishes the construction of $E^1(t)$.

Let $E^2(t)=\dot E^1(t)$ and let $F^2(t)=\dot E^2(t)$. By construction, we have $\omega(F^2(t),E^2(t))=g_t^1(E^1(t),E^1(t))=I$. Since $J(t)$ is Lagrangian, $\omega(E^1(t),E^2(t))=0$. Since $F^2(t)$ is contained in $J^1(t)$ and $E^1(t)$ is contained in $J^{-1}(t)$, $\omega(F^2(t),E^1(t))=0$. By construction, we also have $\omega(F^2(t),F^2(t))=\omega(\ddot E^1(t),\ddot E^1(t))=0$. Next, we complete $E^1(t),E^2(t)$ to a basis of $J(t)$ by adding $E^3(t)$. Moreover, we can assume that $E^3(t)$  satisfies the conditions $g_t^0(E^3(t),E^3(t))=I$, $\omega(E^3(t),E^1(t))=\omega (E^3(t),E^2(t))=0$, $\omega(E^3(t),F^2(t))=0$, and $\omega(E^3(t),\dot F^2(t))=0$. Indeed, let us complete $E^1$, $E^2$ to a basis of $J$ by adding $\bar E^3$. Let $E^3$ be
\[
E^3(t)=O_3(t)(\bar E^3(t)-\omega(\bar E^3(t),F^2(t))E^2(t)-\omega(\bar E^3(t),\dot F^2(t))E^1(t)).
\]

Clearly, we have $\omega(E^3(t),E^1(t))=\omega (E^3(t),E^2(t))=\omega(E^3(t),F^2(t))=0$. We also have
\[
\begin{split}
&\omega(E^3(t),\dot F^2(t))\\
&=O_3(t)\omega(\bar E^3(t),\dot F^2(t))-O_3(t)\omega(\bar E^3(t),F^2(t))\omega(E^2(t),\dot F^2(t))\\
&\quad +O_3(t)\omega(\bar E^3(t),\dot F^2(t))\omega(E^2(t),F^2(t))\\
&=-\omega(\bar E^3(t),F^2(t))\omega(E^2(t),\dot F^2(t))\\
&=\omega(\bar E^3(t),F^2(t))\omega(\ddot E^1(t),\ddot E^1(t))=0
\end{split}
\]

Finally since the kernel of the bilinear form $g^0_t$ is $J^{-1}$, we also obtain
\[
g^0_t(E^3(t),E^3(t))= O_3(t)g^0_t(\bar E^3(t),\bar E^3(t))O_3(t)^T.
\]
Since $g^0_t(\bar E^3(t),\bar E^3(t))$ is positive definite symmetric, we have
\[
g^0_t(E^3(t),E^3(t))=I
\]
if we set $O_3(t)=g^0_t(\bar E^3(t),\bar E^3(t))^{-1/2}$.

Next, we show that $E^3$ can be chosen such that $\ddot E^3(t)$ is contained in $J(t)$. Moreover any such $E^3$ is unique up to  multiplication by an orthogonal matrix (independent of time $t$). Indeed, let $\bar E^3$ be a family defined above. Since $\omega(\dot{\bar E}^3(t),E^1(t))=0$ Then we have
\[
\omega(\ddot{\bar E}^3(t),E^1(t))=-\omega(\dot{\bar E}^3(t),E^2(t))=\omega(\bar E^3(t),F^2(t))=0.
\]
Similarly, since $\omega(\dot{\bar E}^3(t),E^2(t))=0$, we also have
\[
\omega(\ddot{\bar E}^3(t),E^2(t))=-\omega(\dot{\bar E}^3(t),F^2(t))=\omega(\bar E^3(t),\dot F^2(t))=0.
\]

Let $E^3(t)=O(t)\bar E^3(t)$. Then
\[
\begin{split}
\omega(\ddot E^3(t), E^3(t))&=\omega(\ddot O(t)\bar E^3(t)+2\dot O(t)\dot{\bar E}^3(t)+O(t)\ddot{\bar E}^3(t), O(t)\bar E^3(t))\\
&=2\dot O(t)O(t)^T+O(t)\omega(\ddot{\bar E}^3(t), \bar E^3(t))O(t)^T\\
&=2\dot O(t)O(t)^T-O(t)\omega(\dot{\bar E}^3(t), \dot{\bar E}^3(t))O(t)^T.
\end{split}
\]
Therefore, $E^3$ satisfies $\omega(\ddot E^3,E^3)=0$ if and only if $O$ is a solution of the equation $\dot O=\frac{1}{2}O\omega(\dot E^3,\dot E^3)$. This finishes the construction of $E^3$.

Let $F^3(t)=\dot E^3(t)$. We can complete $E^1,E^2,E^3,F^2,F^3$ to a symplectic basis by adding $F^1$. Moreover, there is a unique such $F^1$ satisfying $\omega(\dot F^1(t),F^2(t))=0$. Indeed, suppose we have two ways to complete $E^1,E^2,E^3,F^2,F^3$ to a symplectic basis, say $\bar F^1$ and $F^1$. Then $F^1(t)=\bar F^1(t)+O(t)E^1(t)$ for some matrices $O(t)$. But
\[
\begin{split}
\omega(\dot F^1(t),F^2(t))&=\omega(\dot{\bar F}^1(t)+\dot O(t)E^1(t)+O(t)E^2(t),F^2(t))\\
&=\omega(\dot{\bar F}^1(t),F^2(t))-O(t).
\end{split}
\]
Therefore, $\omega(\dot F^1(t),F^2(t))=0$ if and only if
\[
O=\omega(\dot{\bar F}^1(t),F^2(t)).
\]
\end{proof}

\smallskip

\section{Sasakian manifolds and parallel adapted frames}\label{Sasakian}

In this section, we recall the definition of Sasakian manifolds and introduce the parallel adapted frames. For the part on Sasakian manifolds, we mainly follow \cite{Bl}. Parallel adapted frames were introduced in \cite{Le1}. It will be used to simplify some tedious calculations in a way very similar to the use of geodesic normal coordinates in Riemannian geometry.

Recall that a manifold $M$ of dimension $2n+1$ has an almost contact structure $(\bJ, v_0,\alpha_0)$ if $\bJ:TM\to TM$ is a $(1,1)$ tensor, $v_0$ is a vector field, and $\alpha_0$ is a 1-form satisfying
\[
\bJ^2(v)=-v+\alpha_0(v) v_0\quad\text{and}\quad \alpha_0(v_0)=1
\]
for all tangent vector $v$ in $TM$.

An almost contact structure is normal if the following tensor vanishes
\[
(v,w)\mapsto [\bJ,\bJ](v,w)+d\alpha_0(v,w)v_0,
\]
where $[\bJ,\bJ]$ is defined by
\[
[\bJ,\bJ](v,w)=\bJ^2[v,w]+[\bJ v,\bJ w]-\bJ[\bJ v,w]-\bJ [v,\bJ w].
\]

A Riemannian metric $\left<\cdot,\cdot\right>$ is compatible with a given almost contact manifold if
\[
\left<\bJ v,\bJ w\right>=\left<v,w\right>-\alpha_0(v)\alpha_0(w)
\]
for all tangent vectors $v$ and $w$ in $TM$.

If, in addition, the Riemannian metric satisfies the condition
\[
\left<v,\bJ w\right>=d\alpha_0(v,w),
\]
then we say that the metric is associated to the given almost contact structure.

Finally, a Sasakian manifold is a normal almost contact manifold with an associated Riemannian metric. The following results can be found in \cite{Bl}. Since the sign conventions in \cite{Bl} is different, we include the proof in the appendix.

\begin{thm}\label{Sasaki1}
The followings hold on a Sasakian manifold $(\bJ, v_0,\alpha_0,g=\left<\cdot,\cdot\right>)$
\begin{enumerate}
\item $\mathcal L_{v_0}(\bJ)=0$,
\item $\nabla_{v_0}v_0=0$,
\item $\LD_{v_0}g=0$,
\item $\bJ=-2\nabla v_0$,
\end{enumerate}
where $\nabla$ denotes the Levi-Civita connection.
\end{thm}

\begin{thm}\label{Sasaki2}
An almost contact metric manifold  $(\bJ, v_0,\alpha_0,\left<\cdot,\cdot\right>)$ is Sasakian if and only if it satisfies
\[
(\nabla_v\bJ)w=\frac{1}{2}\left<v,w\right>v_0-\frac{1}{2}\alpha_0(w)v
\]
for all tangent vectors $v$ and $w$.
\end{thm}

Let $\Rm$ denotes the Riemann curvature tensor.

\begin{thm}\label{Sasaki3}
Assume that the almost contact metric manifold  $(\bJ, v_0,\alpha_0,\left<\cdot,\cdot\right>)$ is Sasakian. Then
\[
\begin{split}
\Rm(X,Y)v_0&=\frac{1}{4}\alpha_0(Y)X-\frac{1}{4}\alpha_0(X)Y.
\end{split}
\]
\end{thm}

The Tanaka connection $\nabla^*$ is defined by
\[
\nabla^*_XY=\nabla_XY+\frac{1}{2}\alpha_0(X)\bJ Y-\alpha_0(Y)\nabla_Xv_0+\nabla_X\alpha_0(Y)v_0.
\]
The corresponding curvature operator is denoted by $\Rm^*$ and we call it Tanaka-Webster curvature.

\begin{thm}\label{Sasaki4}
Assume that the tangent vectors $X$, $Y$, and $Z$ are contained in $\ker\alpha_0$. Then
\[
\begin{split}
&\Rm^*(X,Y)Z=(\Rm(X,Y)Z)^h +\left<Z,\nabla_Yv_0\right>\nabla_Xv_0 -\left<Z,\nabla_Xv_0\right>\nabla_Yv_0,
\end{split}
\]
where the superscript $X^h$ denotes the the component of $X$ in $\ker\alpha_0$.

If the manifold is Sasakian, then
\[
\Rm^*(X,Y)Z =(\Rm(X,Y)Z)^h+\frac{1}{4}\left<Z,\bJ Y\right>\bJ X -\frac{1}{4}\left<Z,\bJ X\right>\bJ Y.
\]
\end{thm}

Finally, we introduce the parallel adapted frames.

\begin{lem}\label{parallel1}
Let $v_0$ be a vector field in a Riemannian manifold $M$. Let $\gamma:[0,T]\to M$ be a curve in the Riemannian manifold $M$ and let $v_0,...,v_{2n}$ be an orthonormal frame at $x:=\gamma(0)$. Then there is a orthonormal frame $v_0(t):=v_0(\gamma(t)),v_1(t),...,v_{2n}(t)$ such that
\begin{enumerate}
\item $v_i(0)=v_i$ and
\item $\dot v_i(t)$ is contained in $\Real v_0$ for each $t$,
\end{enumerate}
where $\dot v_i(t)$ denotes the covariant derivative of $v(\cdot)$ along $\gamma(\cdot)$ and $i=1,...,2n$.
\end{lem}

The moving frame defined in Lemma \ref{parallel1} is called parallel adapted frame introduced in \cite{Le1}. Using this frame, we obtain the following convenient local frame.

\begin{lem}\label{parallel2}
Suppose that $(\bJ, v_0,\alpha_0)$ defines an almost contact structure on $M$ and let $\left<\cdot,\cdot\right>$ be an associated Riemannian metric. For each point $x$ in $M$, there is orthonormal frame $v_0,v_1,...,v_{2n}$ defined in a neighborhood of $x$ such that the following conditions hold at $x$.
\begin{enumerate}
\item $\nabla_{v_i}v_j=-\left<\nabla_{v_i}v_0,v_j\right>v_0$,
\item $\nabla_{v_i}v_0=\sum_{j\neq 0}\left<\nabla_{v_i}v_0,v_j\right>v_j$,
\item $\nabla_{v_0}v_i=\nabla_{v_0}v_0=0$,
\end{enumerate}
where $i,j=1,...,2n$.

If, in addition, the manifold $M$ together with $(\bJ, v_0,\alpha_0)$ is Sasakian, then the followings hold at $x$.
\begin{enumerate}
\item $\nabla_{v_i}v_j=\frac{1}{2}\left<\bJ v_i,v_j\right>v_0$,
\item $\nabla_{v_i}v_0=-\frac{1}{2}\bJ v_i$,
\item $\nabla_{v_0}v_i=\nabla_{v_0}v_0=0$.
\end{enumerate}
\end{lem}

The following will be useful for the later sections.

\begin{lem}\label{Chris}
Assume that $(M,\bJ, v_0,\alpha_0,\left<\cdot,\cdot\right>)$ is Sasakian. Let $v_0,v_1,...,v_{2n}$ be a frame defined by Lemma \ref{parallel2}, let $\bJ_{ij}=\left<\bJ v_i,v_j\right>$, and let $\Gamma_{ij}^k=\left<\nabla_{v_i}v_j,v_k\right>$. Then the following holds at $x$
\begin{enumerate}
\item $\Gamma_{00}^i=\Gamma_{0i}^0=\Gamma_{i0}^0=0$,
\item $\Gamma_{ij}^0=-\Gamma_{ji}^0=\frac{1}{2}\bJ_{ij}$,
\item $v_k\bJ_{ij}=0$ if $i,j,k\neq 0$,
\item $\Rm(v_i,v_j)v_k=\sum_{s\neq 0}\left((v_i\Gamma_{jk}^s)-(v_j\Gamma_{ik}^s)-\frac{1}{4}\bJ_{jk}\bJ_{is}+\frac{1}{4}\bJ_{ik}\bJ_{js}\right)v_s$ if $i,j,k\neq 0$.
\end{enumerate}
\end{lem}

\begin{proof}[Proof of Lemma \ref{parallel1}]
Let $w_0(t):=v_0(\gamma(t)),w_1(t),...,w_n(t)$ be an orthonormal frame defined along $\gamma(\cdot)$. Let $O(\cdot)$ be a family of $2n\times 2n$ orthogonal matrices and let $K_{ij}=\left<\dot w_i(t),w_j(t)\right>$, and let $v_i(t):=\sum_{j=1}^{2n} O_{ij}(t)w_j(t)$. By differentiating with respect to time $t$, we have
\[
\left<\dot v_i(t),v_j(t)\right>=\sum_{k,l}\left(\dot O_{ik}(t)+O_{il}(t)K_{lk}(t)\right)O_{jk}(t).
\]
Therefore, by setting $\dot O(t)+O(t)K(t)=0$, we have that $\dot v_i$ is vertical.
\end{proof}

\begin{proof}[Proof of Lemma \ref{parallel2}]
We fix a neighborhood of $x$ on which any point in it can be connected to $x$ by a unique geodesic. We then define $v_i$ to be the vector field on this neighborhood such that $v_i(\gamma(t))$ is a parallel adapted frame along each geodesic $\gamma(\cdot)$ with $\gamma(0)=x$. It follows immediately that
$\nabla_{v_k}v_i$ is vertical, where $i=1,...,2n$ and $k=0,...,2n$. Therefore,
\[
\nabla_{v_k}v_i=\left<\nabla_{v_k}v_i,v_0\right>v_0=-\left<v_i,\nabla_{v_k}v_0\right>v_0.
\]
If $k=0$, then
\[
0=d\alpha_0(v_0,v_i)=-\alpha_0([v_0,v_i])=\left<v_0,\nabla_{v_0}v_i\right>-\left<v_0,\nabla_{v_i}v_0\right>.
\]
Since $|v_0|=1$, we also have
\[
\left<v_0,\nabla_{v_0}v_i\right>=\left<\nabla_{v_i}v_0,v_0\right>=0
\]
and hence $\nabla_{v_0}v_i=0$.

It also follows that $\left<\nabla_{v_0}v_0,v_i\right>=-\left<v_0,\nabla_{v_0}v_i\right>=0$. Therefore, $\nabla_{v_0}v_0=0$. The second part follows from $\left<\nabla_{v_i}v_0,v_j\right>=-\left<\bJ v_i,v_j\right>$ for Sasakian manifolds.
\end{proof}

\begin{proof}[Proof of Lemma \ref{Chris}]
It is clear that $\Gamma_{i0}^0=0$. Since $\nabla_{v_0}v_0=0$,
\[
0=\left<\nabla_{v_0}v_0,v_i\right>=\Gamma_{00}^i=-\Gamma_{0i}^0=0.
\]

Since $\LD_{v_0}g=0$,
\[
0=\LD_{v_0}g(v_i,v_j)=-\left<v_i,[v_0,v_j]\right>-\left<[v_0,v_i],v_j\right>=-\Gamma_{ji}^0-\Gamma_{ij}^0.
\]
Since the Riemannian metric is associated to the almost contact structure,
\[
\bJ_{ji}=\left<v_i,\bJ v_j\right>=d\alpha_0(v_i,v_j)=-\alpha_0([v_i,v_j])=-(\Gamma_{ij}^0-\Gamma_{ji}^0)=2\Gamma_{ji}^0.
\]

The third relation follows from the property of the frame $v_0,...,v_{2n}$ and Theorem \ref{Sasaki2}.

Finally, we have
\[
\begin{split}
&\Rm(v_i,v_j)v_k=\nabla_{v_i}\nabla_{v_j}v_k-\nabla_{v_j}\nabla_{v_i}v_k-\nabla_{[v_i,v_j]}v_k\\
&=\sum_l(v_i\Gamma_{jk}^l)v_l+\sum_{l,s}\Gamma_{jk}^l\Gamma_{il}^sv_s-\sum_l(v_j\Gamma_{ik}^l)v_l\\
&-\sum_{l,s}\Gamma_{ik}^l\Gamma_{jl}^sv_s-\sum_{l,s}\Gamma_{ij}^l\Gamma_{lk}^sv_s+\sum_{l,s}\Gamma_{ji}^l\Gamma_{lk}^sv_s\\
&=\sum_{s\neq 0}\left((v_i\Gamma_{jk}^s)-(v_j\Gamma_{ik}^s)-\frac{1}{4}\bJ_{jk}\bJ_{is}+\frac{1}{4}\bJ_{ik}\bJ_{js}\right)v_s
\end{split}
\]
\end{proof}

\smallskip

\section{Sub-Riemannian geodesic flows and Jacobi curves}

In this section, we give a quick review on some basic notions in sub-Riemannian geometry. In particular, we will introduce Jacobi curves corresponding to the sub-Riemannian geodesic flow and its induced geometric structures.

A sub-Riemannian manifold is a triple $(M,\D,\left<\cdot,\cdot\right>)$, where $M$ is a manifold of dimension $n$, $\D$ is a distribution (sub-bundle of the tangent bundle $TM$), and $\left<\cdot,\cdot\right>$ is a sub-Riemannian metric (smoothly varying inner product defined on $\mathcal D$). Assuming that the manifold $M$ is connected and the distribution $\D$ satisfies the H\"ormander condition (the sections of $\D$ and their iterated Lie brackets span each tangent space, also called ``bracket-generating'' condition). Then, by Chow-Rashevskii Theorem, any two given points on the manifold $M$ can be connected by a horizontal curve (a curve which is almost everywhere tangent to $\D$). Therefore, we can define the sub-Riemannian distance $d$ as
\begin{equation}\label{SRd}
d(x_0,x_1)=\inf_{\gamma\in\Gamma}l(\gamma),
\end{equation}
where the infimum is taken over the set $\Gamma$ of all horizontal paths $\gamma:[0,1]\to M$ satisfying $\gamma(0)=x_0$ and $\gamma(1)=x_1$. The minimizers of (\ref{SRd}) are called length minimizing geodesics (or simply geodesics). As in the Riemannian case, reparametrizations of a geodesic are also geodesics. Therefore, we assume that all geodesics have constant speed. These constant speed geodesics are also minimizers of the kinetic energy functional
\begin{equation}\label{energy}
\inf_{\gamma\in\Gamma}\int_0^1\frac{1}{2}|\dot\gamma(t)|^2dt,
\end{equation}
where $|\cdot|$ denotes the norm w.r.t. the sub-Riemannian metric.

Let $H:T^*M\to\Real$ be the Hamiltonian defined by the Legendre transform:
\[
H(x, p)=\sup_{v\in \D}\left(p(v)-\frac{1}{2}|v|^2\right)
\]
and let
\[
\vec H=\sum_{i=1}^n\left(H_{p_i}\partial_{x_i}-H_{x_i}\partial_{p_i}\right)
\]
be the Hamiltonian vector field. Assume, through out this paper, that the vector field $\vec H$ defines a complete flow which is denoted by $e^{t\vec H}$. The projections of the trajectories of $e^{t\vec H}$ to the manifold $M$ give minimizers of (\ref{energy}).

In this paper, we assume that the sub-Riemannian structure is given by a Sasakian manifold. More precisely, assume that the almost contact structure $(\bJ, v_0,\alpha_0)$ together with the Riemannian structure $\left<\cdot,\cdot\right>$ form a Sasakian manifold. The distribution $\D$ is given by $\D=\ker\alpha_0$ and the sub-Riemannian metric is given by the restriction of the Riemannian metric to $\D$. In this case all minimizers of (\ref{energy}) are given by the projections of the trajectories of $e^{t\vec H}$ (see \cite{Mo} for more detail).

Next, we discuss a sub-Riemannian analogue of Jacobi fields. Let $\omega$ be the symplectic form on the cotangent bundle $T^*M$ defined in local coordinates $(x_1,...,x_{2n+1},p_1,...,p_{2n+1})$ by
\[
\omega=\sum_{i=1}^{2n+1}{dp_i\wedge dx_i}.
\]

Let $\pi : T^*M \rightarrow M$ be the canonical projection and let $\ver$ be the vertical sub-bundle of the cotangent bundle $T^*M$ defined by
\[
\ver_{(x,p)}=\{v\in T_{(x,p)}T^*M| \pi_*(v)=0\}.
\]
The family of Lagrangian subspaces
\begin{equation}\label{Jacobi}
\mathfrak J_{(x,p)}(t):=e^{-t\vec H}_*(\ver_{e^{t\vec H}(x,p)})
\end{equation}
defined a curve in the Lagrangian Grassmannian of $T_{(x,p)}T^*M$, called the Jacobi curve at $(x,p)$ of the flow $e^{t\vec H}$.

Assuming that the manifold is Sasakian. Then Theorem \ref{structure} applies and we let $E^1(t),E^2(t),E^3(t),F^1(t),F^2(t),F^3(t)$ be a canonical frame of $\mathfrak J_{(x,p)}$. This defines a splitting of the vertical space $\ver_{(x,p)}$ and the cotangent space $T_{(x,p)}T^*M$. More precisely, let
\[
\begin{split}
&\ver_1=\text{span}\{E^1(0)\},\quad \ver_2=\text{span}\{E^2(0)\},\quad \ver_3=\text{span}\{E^3(0)\}\\
&\hor_1=\text{span}\{F^1(0)\},\quad \hor_2=\text{span}\{F^2(0)\},\quad \hor_3=\text{span}\{F^3(0)\}.
\end{split}
\]
Then $\ver_{(x,p)}=\ver_1\oplus\ver_2\oplus\ver_3$ and $T_{(x,p)}T^*M=\ver_1\oplus\ver_2\oplus\ver_3\oplus\hor_1\oplus\hor_2\oplus\hor_3$. Note that $\ver_1$, $\ver_2$, $\hor_1$, and $\hor_2$ are all 1-dimensional. $\ver_3$ and $\hor_3$ are $(2n-2)$-dimensional. Let $\alpha$ and $h$ be, respectively, a 1-form and a function on $T^*M$. Let $\vec\alpha$ and $\vec h$ be the vector fields defined, respectively, by
\[
\omega(\vec\alpha,\cdot)=-\alpha\quad \text{ and }\quad \omega(\vec h,\cdot)=-dh.
\]

\begin{thm}\label{split}
Let $x$ be in $M$. The above splitting of the cotangent bundle is given by the followings
\begin{enumerate}
\item $\ver_1=\spn\{\vec\alpha_0\}$,
\item $\ver_2=\spn\{\sum_{k,l\neq 0}h_k\bJ_{kl}\vec\alpha_l\}$,
\item $\ver_3=\spn\{\sum_b a_b\vec\alpha_b|\sum_{j,k\neq 0}a_kh_j\bJ_{kj}=0\text{ and }a_0=\frac{h_0}{2H}\sum_{k\neq 0}a_kh_k\}$,
\item $\hor_1=\spn\{2H\vec h_0-h_0\vec H\}$,
\item $\hor_2=\spn\{h_0\sum_{k\neq 0}h_k\vec\alpha_k-\sum_{j,k\neq 0}h_j\bJ_{jk}\vec h_k-H\vec\alpha_0\\
-\sum_{j,k,l\neq 0}h_jh_l\Gamma_{0l}^k\bJ_{jk}\vec\alpha_0-\sum_{j,k,l,s\neq 0} h_jh_l\bJ_{js}\Gamma_{kl}^s\vec\alpha_k\}$,
\item $\hor_3=\{\sum_{i\neq 0} a_i\vec h_i+\sum_{a}c_a\vec\alpha_a|\sum_{j,k\neq 0}a_kh_j\bJ_{kj}=0,\\
 a_0=\frac{h_0}{2H}\sum_{k\neq 0}a_kh_k, c_0=\sum_{i,j\neq 0}a_ih_j\Gamma_{0j}^i, \\c_k=\sum_{j\neq 0}\left(\frac{1}{2}a_j\bJ_{jk} h_0-\frac{1}{2}a_0h_j\bJ_{jk}+\sum_{i\neq 0}a_ih_j\Gamma_{kj}^i\right) \}$,
\end{enumerate}
where $v_0,v_1,...,v_{2n}$ is a local frame defined in a neighborhood of a point $x$ by Lemma \ref{parallel2}, $\bJ_{ij}=\left<\bJ v_i,v_j\right>$.
\end{thm}

The vertical splitting can be written in a coordinate free way. For this, we identify the tangent bundle $TM$ with the vertical bundle $\ver$ using the Riemannian metric via
\[
v\in TM\to \alpha(\cdot)=\left<v,\cdot\right>\in T^*M\to -\vec\alpha\in ver.
\]

Under this identification, we have
\begin{thm}\label{splitcan}
Let $x$ be in $M$. The above splitting of the cotangent bundle is given by the followings
\begin{enumerate}
\item $\ver_1=\Real v_0$,
\item $\ver_2=\Real \bJ p^h$,
\item $\ver_3=\Real (p^h+p(v_0)v_0)\oplus\{v|\left<v,p^h\right>=\left<v,\bJ p^h\right>=\left<v,v_0\right>=0\}$.
\item $\pi_*\hor_1=\Real(|p^h|^2 v_0-p(v_0)p^h)$,
\item $\pi_*\hor_2=\Real \bJ p^h$,
\item $\pi_*\hor_3=\{X|\left<X,\bJ p^h\right>=\left<X,v_0\right>=0\}$,
\end{enumerate}
where $p^h$ is the vector in $\ker\alpha_0$ defined by $p(v)=\left<p^h,v\right>$ and $v$ ranges over vectors in $\ker\alpha_0$.
\end{thm}

Under the above identification, we can also define a volume form $\mathfrak m$ on $\ver$ by $\mathfrak m(v_0,...,v_{2n})=1$. The Riemannian volume on $M$ is denoted by $\eta$. The proof of Theorem \ref{split} also gives

\begin{thm}\label{vol}
The volume forms $\mathfrak m$ and $\eta$ satisfy
\begin{enumerate}
\item $\mathfrak m(E(0))=\frac{1}{|p^h|}$,
\item $\eta(\pi_*F(0))=|p^h|$.
\end{enumerate}
\end{thm}

\begin{proof}[Proof of Theorem \ref{split}]
Let $v_0,v_1,...,v_{2n}$ be the local frame defined in a neighborhood of $x$ by Lemma \ref{parallel2}. Let $\Gamma_{ab}^c$ and $\bJ_{ij}$ be defined by
\[
\nabla_{v_a}v_b=\Gamma_{ab}^cv_c\quad\text{and}\quad \bJ_{ij}=\left<\bJ v_i,v_j\right>,
\]
respectively. From now on, we sum over repeated indices. The indices $i,j,k,s,l$ ranges over $1,...,2n$ and $a,b,c,d$ ranges over $0,...,2n$.

It is clear that $\Gamma_{ab}^c=-\Gamma_{ac}^b$ wherever it is defined. We also have $\Gamma_{00}^i=\Gamma_{0i}^0=\Gamma_{i0}^0=0$. Indeed, since $d\alpha_0(v_0,v_i)=0$, we have
\[
0=\alpha_0([v_0,v_i])=\Gamma_{0i}^0-\Gamma_{i0}^0=\Gamma_{0i}^0=-\Gamma_{00}^i.
\]

Since $\left<\bJ v_i,v_j\right>=-2\left<\nabla_{v_i}v_0,v_j\right>$, we have $\bJ_{ij}=-2\Gamma_{i0}^j=2\Gamma_{ij}^0$. Let $\alpha_0,...,\alpha_{2n}$ be the dual frame of $v_0,...,v_{2n}$ and let $h_i(x,p)=p(v_i)$. Then $\pi^*\alpha_0,...,\pi^*\alpha_n,dh_0,...,dh_n$ forms a local co-frame of the cotangent bundle. We will also denote $\pi^*\alpha_i$ simply by $\alpha_i$. 

The proof of the following two lemmas will be postponed to the appendix.
\begin{lem}\label{relation}
The following relations hold.
\begin{enumerate}
\item $\alpha_a(\vec h_b)=\delta_{ab}$,
\item $[\vec\alpha_a,\vec\alpha_b]=0$,
\item $dh_b(\vec h_c)=\sum_a(\Gamma_{cb}^a-\Gamma_{bc}^a)h_a$,
\item $[\vec\alpha_a,\vec h_b]=\sum_c(\Gamma_{bc}^a-\Gamma_{cb}^a)\vec\alpha_c$,
\item $[\vec H,\vec\alpha_i]=\vec h_i+\sum_{j\neq 0,a}h_j(\Gamma_{aj}^i-\Gamma_{ja}^i)\vec\alpha_a$ if $i\neq 0$,
\item $[\vec H,\vec\alpha_0]=\sum_{j,k\neq 0}h_j(\Gamma_{kj}^0-\Gamma_{jk}^0)\vec\alpha_k=-\sum_{j,k\neq 0}h_j\bJ_{jk}\vec\alpha_k$,
\item $[\vec H,\vec h_i]=\sum_{k\neq 0}h_k[\vec h_k,\vec h_i]-\sum_{k\neq 0,a}h_a(\Gamma_{ik}^a-\Gamma_{ki}^a)\vec h_k$,
\item $[\vec H,[\vec H,\vec\alpha_0]]=h_0\sum_{k\neq 0}h_k\vec\alpha_k-\sum_{k,j\neq 0}h_j\bJ_{jk}\vec h_k\\
    -H\vec\alpha_0-\sum_{j,l,k\neq 0}h_jh_l\Gamma_{0l}^k\bJ_{jk}\vec\alpha_0-\sum_{j,l,s,k\neq 0}h_jh_l\bJ_{js}\Gamma_{kl}^s\vec\alpha_k$,
\item $[\vec H,[\vec H,\vec\alpha_i]]=2\sum_{l,k\neq 0}h_l\Gamma_{li}^k\vec h_k+\sum_{l\neq 0}h_l\bJ_{li}\vec h_0-\sum_{k\neq 0}h_0\bJ_{ik}\vec h_k$ (mod vertical) when $i\neq 0$,
\item $[\vec H,[\vec H,[\vec H,\vec\alpha_0]]]=h_0\vec H-2H\vec h_0$ (mod vertical).
\end{enumerate}
Here, the phrase ``mod vertical'' means the that the difference of the two vectors is contained in the vertical bundle $\ver$.
\end{lem}

The relations reduce to the following ones at $x$
\begin{lem}\label{relationx}
The following relations hold at $x$.
\begin{enumerate}
\item $dh_j(\vec h_i)=\bJ_{ij}h_0$ if $i\neq 0\neq j$,
\item $dh_j(\vec h_0)=\frac{1}{2}\sum_{k\neq 0}\bJ_{jk}h_k$ if $j\neq 0$,
\item $[\vec\alpha_i,\vec h_j]=\frac{1}{2}\bJ_{ij}\vec\alpha_0$ if $i\neq 0\neq j$,
\item $[\vec\alpha_i,\vec h_0]=\frac{1}{2}\sum_{k\neq 0}\bJ_{ki}\vec\alpha_k$ if $i\neq 0$,
\item $[\vec\alpha_0,\vec h_j]=\sum_{k\neq 0}\bJ_{jk}\vec\alpha_k$ if $j\neq 0$,
\item $[\vec H,\vec\alpha_i]=\vec h_i+\sum_{j\neq 0}h_j\bJ_{ji}\vec\alpha_0$ when $i\neq 0$,
\item $[\vec H,\vec\alpha_0]=-\sum_{j,k\neq 0}h_j\bJ_{jk}\vec\alpha_k$,
\item $[\vec H,[\vec H,\vec\alpha_0]]=h_0\sum_{k\neq 0}h_k\vec\alpha_k-\sum_{j,k\neq 0}h_j\bJ_{jk}\vec h_k-H\vec\alpha_0$,
\end{enumerate}
\end{lem}

Now, we apply the above lemmas to prove the theorem. Since $[\vec H,\vec\alpha_0]$ is vertical, $\vec\alpha_0$ is in $J^{-1}(0)$. Therefore, $\vec\alpha_0=fE^1(0)$ for some function $f$ on the cotangent bundle. It follows from Theorem \ref{structure} that
\begin{enumerate}
\item $fE^2(0)=[\vec H,\vec\alpha_0]-(\vec H f)E^1(0)$,
\item $fF^2(0)=[\vec H,[\vec H,\vec\alpha_0]]-(\vec H^2f)E^1(0)-2(\vec H f)E^2(0)$,
\item $f\dot F^2(0)=[\vec H,[\vec H,[\vec H,\vec\alpha_0]]]-(\vec H^3 f)E_1-3(\vec H^2 f)E_2-3(\vec H f)F_2$.
\end{enumerate}

By Lemma \ref{relationx}, we have
\[
f^2=\omega(fF^2(0),fE^2(0))=\sum_{i,l,j,k\neq 0}h_ih_j\bJ_{il}\bJ_{jk}\omega(\vec h_l,\vec\alpha_k)=2H.
\]
It follows from this and Lemma \ref{relation} that
\begin{enumerate}
\item $fE^2(0)=-\sum_{k,l\neq 0}h_k\bJ_{kl}\vec\alpha_l$,
\item $fF^2(0)=h_0\sum_{j,k,l\neq 0}h_k\vec\alpha_k-\sum_{j,k\neq 0}h_j\bJ_{jk}\vec h_k-H\vec\alpha_0\\
-\sum_{j,k,l\neq 0}h_jh_l\Gamma_{0l}^k\bJ_{jk}\vec\alpha_0-\sum_{j,k,l,s\neq 0} h_jh_l\bJ_{js}\Gamma_{kl}^s\vec\alpha_k$,
\item $-fF^1(0)=f\dot F^2(0)=h_0\vec H-2H\vec h_0$ (mod vertical).
\end{enumerate}
This gives the characterizations of $\ver_1$, $\ver_2$, and $\hor_2$.

Suppose that $a_b\vec\alpha_b$ is contained in $\ver_3$. Since $\ver_3$ and $\hor_2$ are skew-orthogonal,
\begin{equation}\label{H3-1}
\begin{split}
&-\sum_{j,k\neq 0}a_kh_j\bJ_{kj}=\omega\left(a_b\vec\alpha_b,h_j\bJ_{ij}\vec h_i\right)=0.
\end{split}
\end{equation}

Since $\ver_3$ and $\hor_1$ are skew-orthogonal, we also have
\begin{equation}\label{V3-1}
\begin{split}
0&=-\omega\left(a_b\vec\alpha_b,h_0\vec H-2H\vec h_0\right)=h_0h_ka_k-2Ha_0\\
\end{split}
\end{equation}
This gives the characterizations of $\ver_3$.

It also follows that
\[
\begin{split}
&[\vec H,a_0\vec\alpha_0+a_i\vec\alpha_i]\\
&=(\vec Ha_0)\vec\alpha_0+a_0[\vec H,\vec\alpha_0]+(\vec Ha_i)\vec\alpha_i+a_i[\vec H,\vec\alpha_i]\\
&=(\vec Ha_0)\vec\alpha_0-a_0h_j\bJ_{jk}\vec\alpha_k+(\vec Ha_i)\vec\alpha_i+a_i\vec h_i+a_ih_j(\Gamma_{aj}^i-\Gamma_{ja}^i)\vec\alpha_a.
\end{split}
\]

It follows from the structural equation that $[\vec H,a_0\vec\alpha_0+a_i\vec\alpha_i]$ is contained in $\ver_3\oplus\hor_3$. Moreover, if $X_1$ and $X_2$ are the $\ver_3$ and $\hor_3$ parts of $[\vec H,a_0\vec\alpha_0+a_i\vec\alpha_i]$, respectively, then
\[
\pi_*[\vec H,X_1]=\pi_*[\vec H,X_2].
\]

Suppose that $a_i\vec h_i+c_a\vec\alpha_a$ is contained in $\hor_3$. Then it follows from Lemma \ref{relation} and the characterization of $\ver_3$ that
\[
\begin{split}
&\pi_*[\vec H,a_i\vec h_i+ c_a\vec\alpha_a]\\
&= (\vec Ha_i)v_i+a_ih_j[v_j,v_i]-a_i\bJ_{ik} h_0v_k-a_ih_j(\Gamma_{ik}^j-\Gamma_{ki}^j)v_k+c_iv_i\\
&=(\vec Ha_i)v_i+a_ih_j(\Gamma_{ji}^k-\Gamma_{ij}^k)v_k-a_i\bJ_{ik} h_0v_k-a_ih_j(\Gamma_{ik}^j-\Gamma_{ki}^j)v_k+c_iv_i\\
&=(\vec Ha_i)v_i+ a_ih_j(\Gamma_{ji}^k+\Gamma_{ki}^j)v_k- a_i\bJ_{ik} h_0v_k+ c_iv_i\\
\end{split}
\]
and
\[
\begin{split}
&\pi_*[\vec H,(\vec Ha_0)\vec\alpha_0-a_0h_j\bJ_{jk}\vec\alpha_k+(\vec Ha_i)\vec\alpha_i+a_ih_j(\Gamma_{aj}^i-\Gamma_{ja}^i)\vec\alpha_a-c_a\vec\alpha_a]\\
&=-a_0h_j\bJ_{jk}v_k+(\vec Ha_i)v_i+a_ih_j(\Gamma_{kj}^i-\Gamma_{jk}^i)v_k-c_iv_i\\
\end{split}
\]

It follows that
\[
\begin{split}
&c_k=a_ih_j\Gamma_{kj}^i+\frac{1}{2}(a_j\bJ_{jk} h_0-a_0h_j\bJ_{jk}).
\end{split}
\]

It also follows from this that
\[
\begin{split}
&(\vec Ha_0-c_0)\vec\alpha_0-a_0h_j\bJ_{jk}\vec\alpha_k+(\vec Ha_i)\vec\alpha_i+a_ih_j(\Gamma_{0j}^i-\Gamma_{j0}^i)\vec\alpha_0\\
&+a_ih_j(\Gamma_{kj}^i-\Gamma_{jk}^i)\vec\alpha_k-\left(\frac{1}{2}a_j\bJ_{jk} h_0-\frac{1}{2}a_0h_j\bJ_{jk}+a_ih_j\Gamma_{kj}^i\right)\vec\alpha_k\\
&=(\vec Ha_0-c_0+a_ih_j\Gamma_{0j}^i)\vec\alpha_0+(\vec Ha_i)\vec\alpha_i-a_ih_j\Gamma_{jk}^i\vec\alpha_k-\frac{1}{2}\left(a_j\bJ_{jk} h_0+a_0h_j\bJ_{jk}\right)\vec\alpha_k
\end{split}
\]
is contained in $\ver_3$. Therefore,
\[
\begin{split}
&2H\left(\vec Ha_0-c_0+a_ih_j\Gamma_{0j}^i\right)\\
&=h_0\left(\vec Ha_k-\frac{1}{2}a_jh_0\bJ_{jk} -\frac{1}{2}a_0h_j\bJ_{jk}-a_ih_j\Gamma_{jk}^i\right)h_k\\
&=h_0\left(\vec Ha_k\right)h_k-h_0a_i\Gamma_{jk}^ih_jh_k\\
\end{split}
\]

On the other hand, it follows from (\ref{V3-1}) that
\[
h_0h_lh_s\Gamma_{lk}^sa_k+h_0h_k\vec Ha_k-2H\vec Ha_0=0.
\]

Therefore, $c_0=a_ih_j\Gamma_{0j}^i$ and this finishes the characterization of $\hor_3$.

By the tenth relation in Lemma \ref{relationx} and the structural equation, we can choose a vector in $\hor_1$ of the form
\[
2H\vec h_0-h_0\vec H+r_a\vec\alpha_a.
\]

Since $\hor_1$ is in the skew orthogonal complement of $\hor_3$, we have
\[
\begin{split}
0&=\omega\left(a_i\vec h_i+c_a\vec\alpha_a,2H\vec h_0-h_0\vec H+r_a\vec\alpha_a\right)\\
&=2Ha_idh_0(\vec h_i)-2Hc_0-h_0a_idH(\vec h_i)+h_0c_jh_j+r_ia_i\\
&=-2Ha_i\Gamma_{0i}^sh_s-2Hc_0+h_0a_ih_jh_k\Gamma_{ji}^k+h_0c_jh_j+r_ia_i\\
&=-2Ha_i\Gamma_{0i}^sh_s-2Hc_0-h_0a_ih_jh_k\Gamma_{jk}^i+h_0a_ih_jh_k\Gamma_{kj}^i+r_ia_i\\
&=r_ia_i.
\end{split}
\]

Therefore, by (\ref{H3-1}), we have $r_i=r\bJ_{ij}h_j$ for some $r$, where $i=1,...,2n$.

Since $\hor_2$ is also skew orthogonal to $\hor_1$, we also have
\[
\begin{split}
0&=\omega\Big(h_0h_k\vec\alpha_k-h_j\bJ_{jk}\vec h_k-H\vec\alpha_0-h_jh_l\Gamma_{0l}^k\bJ_{jk}\vec\alpha_0\\
&- h_jh_l\bJ_{js}\Gamma_{kl}^s\vec\alpha_k,2H\vec h_0-h_0\vec H+r_0\vec\alpha_0+r\bJ_{ij}h_j\vec\alpha_i\Big)\\
&=-2Hdh_0\Big(h_j\bJ_{jk}\vec h_k+H\vec\alpha_0+h_jh_l\Gamma_{0l}^k\bJ_{jk}\vec\alpha_0\Big)\\
&-h_0dH\Big(h_0h_k\vec\alpha_k-h_j\bJ_{jk}\vec h_k-h_jh_l\bJ_{js}\Gamma_{kl}^s\vec\alpha_k\Big)-r\bJ_{ij}h_j\alpha_i\Big(h_l\bJ_{lk}\vec h_k\Big)\\
&=-2Hh_j\bJ_{jk}dh_0(\vec h_k)+(2H)^2+2Hh_jh_l\Gamma_{0l}^k\bJ_{jk}\\
&+4h_0^2H+h_0h_jh_l\bJ_{jk}dh_l(\vec h_k)-h_0 h_ih_jh_l\bJ_{js}\Gamma_{il}^s+2rH\\
&=2rH.
\end{split}
\]
Therefore $r=0$.
Finally, since $2H\vec h_0-h_0\vec H+r_0\vec\alpha_0$ is in $\hor_1$, it follows from the structural equation that
\[
\begin{split}
0&=\omega([\vec H,2H\vec h_0-h_0\vec H+r_0\vec\alpha_0],2h_0h_k\vec\alpha_k-h_j\bJ_{jk}\vec h_k-2H\vec\alpha_0)\\
&=r_0\omega([\vec H,\vec\alpha_0],2h_0h_k\vec\alpha_k-h_j\bJ_{jk}\vec h_k-2H\vec\alpha_0).
\end{split}
\]
Hence, $r_0=0$ and this gives $\hor_1$.

\end{proof}

\smallskip

\section{Curvatures of sub-Riemannian geodesic flows}\label{curvature}

In this section, we will focus on the computation of the curvature $R^{ij}(0)$, where the Jacobi curve is given by the sub-Riemannian geodesic flow. For this, let $\mathcal R^{ij}:\ver_i\to\ver_j$ be the operator for which the matrix representation with respect to bases $E^i(0)$ and $E^j(0)$ of $\ver_i$ and $\ver_j$, respectively, is given by $R^{ij}(0)$. More precisely,
\[
\mathcal R^{ij}(E^i_k(0))=\sum_l R^{ij}_{kl}(0)E_l^j(0),
\]
where $R^{ij}_{kl}(0)$ is the $kl$-th entry of $R^{ij}(0)$.

\begin{thm}\label{compute}
Assume that the manifold is Sasakian. Then, under the identifications of Theorem \ref{splitcan}, $\mathcal R$ is given by
\begin{enumerate}
\item $\mathcal R(v)=0$ for all $v$ in $\ver_1$,
\item $\mathcal R(v)_{\ver_2}=(\Rm(\bJ p^h,p^h)p^h)_{\ver_2}+\left(\frac{1}{4}|p^h|^2+p(v_0)^2\right)\bJ p^h\\=(\Rm^*(\bJ p^h,p^h)p^h)_{\ver_2}+p(v_0)^2\bJ p^h$ for all $v$ in $\ver_2$,
\item $\mathcal R(v)_{\ver_3}=(\Rm(\bJ p^h,p^h)p^h)_{\ver_3}=(\Rm^*(\bJ p^h,p^h)p^h)_{\ver_3}$ for all $v$ in $\ver_2$,
\item $\mathcal R(v)_{\ver_1}=0$ for all $v$ in $\ver_3$,
\item $\mathcal R(v)_{\ver_2}=(\Rm(v^h,p^h)p^h)_{\ver_2}=(\Rm^*(\bJ p^h,p^h)p^h)_{\ver_2}$ for all $v$ in $\ver_3$,
\item $\mathcal R(p^h+p(v_0)v_0)=0$,
\item $\mathcal R(v)_{\ver_3}=(\Rm(v^h,p^h)p^h)_{\ver_3}+\frac{1}{4}p(v_0)^2v^h=(\Rm^*(v^h,p^h)p^h)_{\ver_3}\\+\frac{1}{4}p(v_0)^2v^h$ for all $v$ in $\ver_3$ satisfying $\left<v^h,p^h\right>=0$.
\end{enumerate}
\end{thm}

\begin{proof}
Let $\Lambda_{\ver_i\hor_j}:\ver_i\to\hor_j$ be the operator defined by
\[
\Lambda_{\ver_i\hor_j}(V)=[\vec H,V]_{\hor_j},
\]
where $V$ is a section in $\ver_i$ and the subscript $\hor_j$ denotes the $\hor_j$-component of the vector.

It follows from (\ref{structural}) that $\Lambda_{\ver_i\hor_j}$ is tensorial and so well-defined. We also define operators $\Lambda_{\ver_i\ver_j}$, $\Lambda_{\hor_i\ver_j}$, and $\Lambda_{\hor_i\hor_j}$ in a similar way. By (\ref{structural}), we have

\begin{lem}
The following relations hold.
\begin{enumerate}
\item $\mathcal R^{11}=\Lambda_{\hor_1\ver_1}\circ\Lambda_{\hor_2\hor_1}\circ\Lambda_{\ver_2\hor_2}\circ\Lambda_{\ver_1\ver_2}$,
\item $\mathcal R^{13}=\Lambda_{\hor_1\ver_3}\circ\Lambda_{\hor_2\hor_1}\circ\Lambda_{\ver_2\hor_2}\circ\Lambda_{\ver_1\ver_2}$,
\item $\mathcal R^{22}=-\Lambda_{\hor_2\ver_2}\circ\Lambda_{\ver_2\hor_2}$,
\item $\mathcal R^{23}=-\Lambda_{\hor_2\ver_3}\circ\Lambda_{\ver_2\hor_2}$,
\item $\mathcal R^{31}=-\Lambda_{\hor_3\ver_1}\circ\Lambda_{\ver_3\hor_3}$,
\item $\mathcal R^{32}=-\Lambda_{\hor_3\ver_2}\circ\Lambda_{\ver_3\hor_3}$,
\item $\mathcal R^{33}=-\Lambda_{\hor_3\ver_3}\circ\Lambda_{\ver_3\hor_3}$.
\end{enumerate}
\end{lem}

Clearly, $\Lambda_{\hor_1\ver_1}\equiv 0$ and $\Lambda_{\hor_1\ver_3}\equiv 0$. For the rest, we need a lemma for which the proof is given in the appendix.

\begin{lem}\label{relationxx}
The following holds at $x$
\begin{enumerate}
\item $[\vec h_k,\vec h_i]=\bJ_{ki}\vec h_0+\sum_a b_{ki}^a\vec\alpha_a$,
\item $\sum_{k\neq 0}h_kb_{ki}^0=\sum_{k,s\neq 0}h_kh_sv_k(\Gamma_{0i}^s)$ if $k,i\neq 0$,
\item $\sum_{k\neq 0}h_kb_{ki}^l=-\sum_{s,k\neq 0}h_sh_k[v_k\Gamma_{il}^s-v_k\Gamma_{li}^s-v_i\Gamma_{kl}^s]$ if $k,i,l\neq 0$,
\item $[\vec H,\vec h_i]=\sum_{k\neq 0}h_k\bJ_{ki}\vec h_0-\sum_{k\neq 0}h_0\bJ_{ik}\vec h_k+\sum_{k\neq 0,a}h_k b_{ki}^a\vec\alpha_a$.
\end{enumerate}
\end{lem}

Let $a_i\vec h_i+c_a\vec\alpha_a$ be a vector in $\hor_3$. A computation shows that the followings hold at $x$.
\[
\begin{split}
&[\vec H, a_i\vec h_i+c_a\vec\alpha_a]\\
&=(\vec Ha_i)\vec h_i+(\vec Hc_0)\vec\alpha_0+(\vec Hc_i)\vec\alpha_i+a_i[\vec H,\vec h_i]+c_0[\vec H,\vec\alpha_0]+c_i[\vec H,\vec\alpha_i]\\
&=(\vec Ha_i)\vec h_i+a_ih_jh_l(v_l\Gamma_{0j}^i)\vec\alpha_0+(\vec Hc_i)\vec\alpha_i-h_0a_i\bJ_{ik}\vec h_k\\
&+a_ih_kh_s(v_k\Gamma_{0i}^s)\vec\alpha_0+a_ih_k b_{ki}^j\vec\alpha_j+c_k(\vec h_k+h_j\bJ_{jk}\vec\alpha_0)\\
&= (\vec Ha_k)\vec h_k-\frac{1}{2}(a_j\bJ_{jk} h_0+a_0h_j\bJ_{jk})\vec h_k+(\vec Hc_i)\vec\alpha_i+a_ih_k b_{ki}^j\vec\alpha_j.
\end{split}
\]

On the other hand, we have
\[
\frac{h_0}{2H}\left(\vec Ha_k-\frac{1}{2}a_j\bJ_{jk} h_0-\frac{1}{2}a_0h_j\bJ_{jk}\right)h_k=\frac{h_0}{2H}(\vec Ha_k)h_k
\]
and
\[
\begin{split}
&\frac{1}{2}\left(\vec Ha_i-\frac{1}{2}a_j\bJ_{ji} h_0-\frac{1}{2}a_0h_j\bJ_{ji}\right)\bJ_{ik}h_0-\frac{1}{2}\frac{h_0}{2H}(\vec Ha_i)h_ih_j\bJ_{jk}\\
&=\frac{1}{2}(\vec Ha_i)\bJ_{ik}h_0+\frac{1}{4}a_kh_0^2+\frac{1}{4}a_0h_kh_0-\frac{h_0}{4H}(\vec Ha_i)h_ih_j\bJ_{jk}\\
\end{split}
\]
at $x$.

Therefore,
\[
\begin{split}
&[\vec H, a_i\vec h_i+c_a\vec\alpha_a]_\ver= -\frac{1}{2}(\vec Ha_i)\bJ_{ik}h_0\vec\alpha_k-\frac{1}{4}a_kh_0^2\vec\alpha_k\\
&-\frac{1}{4}a_0h_kh_0\vec\alpha_k+\frac{h_0}{4H}(\vec Ha_i)h_ih_j\bJ_{jk}\vec\alpha_k+(\vec Hc_k)\vec\alpha_k+a_ih_j b_{ji}^k\vec\alpha_k.
\end{split}
\]

Another computation shows that
\[
\begin{split}
&\vec H c_k=\frac{1}{2}(\vec H a_j)\bJ_{jk} h_0-\frac{1}{2}(\vec Ha_0)h_j\bJ_{jk}-\frac{1}{2}a_0h_ldh_j(\vec h_l)\bJ_{jk}+a_ih_jh_l(v_l\Gamma_{kj}^i)\\
&=\frac{1}{2}(\vec H a_j)\bJ_{jk} h_0-\frac{h_0}{4H}(\vec Ha_l)h_lh_j\bJ_{jk}\\
&-\frac{h_0}{4H}a_lh_sdh_l(\vec h_s)h_j\bJ_{jk}-\frac{1}{2}a_0h_ldh_j(\vec h_l)\bJ_{jk}+a_ih_jh_l(v_l\Gamma_{kj}^i)\\
&=\frac{1}{2}(\vec H a_j)\bJ_{jk} h_0-\frac{h_0}{4H}(\vec Ha_l)h_lh_j\bJ_{jk}+\frac{1}{2}a_0h_0h_k+a_ih_jh_l(v_l\Gamma_{kj}^i)
\end{split}\]

Hence,
\[
\begin{split}
&[\vec H, a_i\vec h_i+c_a\vec\alpha_a]_\ver= -\frac{1}{2}(\vec Ha_i)\bJ_{ik}h_0\vec\alpha_k-\frac{1}{4}a_kh_0^2\vec\alpha_k -\frac{1}{4}a_0h_kh_0\vec\alpha_k\\
&+\frac{h_0}{4H}(\vec Ha_i)h_ih_j\bJ_{jk}\vec\alpha_k+\frac{1}{2}(\vec H a_j)\bJ_{jk} h_0\vec\alpha_k-\frac{h_0}{4H}(\vec Ha_l)h_lh_j\bJ_{jk}\vec\alpha_k\\
&+\frac{1}{2}a_0h_0h_k\vec\alpha_k+a_ih_jh_l(v_l\Gamma_{kj}^i)\vec\alpha_k+a_ih_j b_{ji}^k\vec\alpha_k\\
&= -\frac{1}{4}a_kh_0^2\vec\alpha_k +\frac{1}{4}a_0h_kh_0\vec\alpha_k-a_ih_sh_l(v_l\Gamma_{ik}^s-v_i\Gamma_{lk}^s)\vec\alpha_k\\
&= -\frac{1}{4}h_0(a_kh_0-a_0h_k)\vec\alpha_k -a_ih_sh_l\Rm_{ilsk}\vec\alpha_k.
\end{split}
\]
where $\Rm_{ijks}=\left<\Rm(v_i,v_j)v_k,v_s\right>$.

This finishes the proof of the last four assertions. Let
\[
\begin{split}
&h_j\bJ_{jk}\vec h_k-h_0h_k\vec\alpha_k+H\vec\alpha_0+h_jh_l\Gamma_{0l}^k\bJ_{jk}\vec\alpha_0+h_jh_l\bJ_{js}\Gamma_{kl}^s\vec\alpha_k
\end{split}
\]
be a section of the bundle $\hor_2$. Then
\[
\begin{split}
&\Big[\vec H,h_j\bJ_{jk}\vec h_k-h_0h_k\vec\alpha_k+H\vec\alpha_0+h_jh_l\Gamma_{0l}^k\bJ_{jk}\vec\alpha_0+h_jh_l\bJ_{js}\Gamma_{kl}^s\vec\alpha_k\Big]\\
&=h_idh_j(\vec h_i)\bJ_{jk}\vec h_k-h_0h_idh_k(\vec h_i)\vec\alpha_k+H[\vec H,\vec\alpha_0]+h_j\bJ_{jk}[\vec H,\vec h_k]\\
&-h_0h_k[\vec H,\vec\alpha_k]+h_jh_lh_i(v_i\Gamma_{0l}^k)\bJ_{jk}\vec\alpha_0+h_jh_lh_i(v_i\Gamma_{kl}^s)\bJ_{js}\vec\alpha_k\\
&=-2h_0\vec H-h_0^2h_i\bJ_{ik}\vec\alpha_k-Hh_j\bJ_{jk}\vec\alpha_k+h_j\bJ_{jk}(h_i\bJ_{ik}\vec h_0-h_0\bJ_{ki}\vec h_i)\\
&+h_jh_i\bJ_{jk} b_{ik}^a\vec\alpha_a+h_jh_lh_i(v_i\Gamma_{0l}^k)\bJ_{jk}\vec\alpha_0+h_jh_lh_i(v_i\Gamma_{kl}^s)\bJ_{js}\vec\alpha_k\\
&=2H\vec h_0-h_0\vec H-h_0^2h_i\bJ_{ik}\vec\alpha_k-Hh_j\bJ_{jk}\vec\alpha_k\\
&+h_jh_i\bJ_{jk} b_{ik}^a\vec\alpha_a+h_jh_lh_i(v_i\Gamma_{0l}^k)\bJ_{jk}\vec\alpha_0+h_jh_lh_i(v_i\Gamma_{kl}^s)\bJ_{js}\vec\alpha_k.
\end{split}
\]

It follows that
\[
\begin{split}
&\Big[\vec H,h_j\bJ_{jk}\vec h_k-h_0h_k\vec\alpha_k+H\vec\alpha_0+h_jh_l\Gamma_{0l}^k\bJ_{jk}\vec\alpha_0+h_jh_l\bJ_{js}\Gamma_{kl}^s\vec\alpha_k\Big]_\ver\\
&=-h_0^2h_i\bJ_{ik}\vec\alpha_k-Hh_j\bJ_{jk}\vec\alpha_k\\
&+h_jh_i\bJ_{jk} b_{ik}^a\vec\alpha_a+h_jh_lh_i(v_i\Gamma_{0l}^k)\bJ_{jk}\vec\alpha_0+h_jh_lh_i(v_i\Gamma_{kl}^s)\bJ_{js}\vec\alpha_k\\
&=-h_0^2h_i\bJ_{ik}\vec\alpha_k-Hh_j\bJ_{jk}\vec\alpha_k+h_jh_ih_s\bJ_{jk}(v_i\Gamma_{0k}^s)\vec\alpha_0\\
&+h_jh_lh_i(v_i\Gamma_{0l}^k)\bJ_{jk}\vec\alpha_0-h_sh_jh_k\bJ_{ji}(v_k\Gamma_{il}^s-v_k\Gamma_{li}^s-v_i\Gamma_{kl}^s+v_k\Gamma_{li}^s)\vec\alpha_l\\
&=-h_0^2h_i\bJ_{ik}\vec\alpha_k-Hh_j\bJ_{jk}\vec\alpha_k-h_sh_jh_k\bJ_{ji}(v_k\Gamma_{il}^s-v_i\Gamma_{kl}^s)\vec\alpha_l\\
&=-\left(h_0^2+\frac{1}{2}H\right)h_i\bJ_{ik}\vec\alpha_k -h_jh_kh_s\bJ_{ji}\Rm_{kils}\vec\alpha_l.
\end{split}
\]
\end{proof}

\smallskip

\section{Conjugate time estimates and Bonnet-Myer's type theorem}\label{conjestimate}

In this section, we give estimates for the first conjugate time under certain curvature lower bound. Let $\psi_t:T^*_xM\to M$ be the map defined by $\psi_t(x,p)=\pi(e^{t \vec H}(x,p))$, where $\pi:T^*M\to M$ is the projection. Let us fix a covector $(x,p)$. The first conjugate time is the smallest $t_0>0$ such that the linear map $(d\psi_{t_0})_{(x,p)}$ is not bijective. The curve $t\mapsto\psi_t(x,p)$ is no longer minimizing if $t>t_0$ (see \cite{AgSa}).

\begin{thm}\label{conj}
Assume that the Tanaka-Webster curvature $\Rm^*$ of the Sasakian manifold satisfies
\begin{enumerate}
\item $\left<\Rm^*(\bJ p^h,p^h)p^h,\bJ p^h\right>\geq k_1|p^h|^4$,
\item $\sum_{i=1}^{2n-2}\left<\Rm^*(w_i,p^h)p^h,w_i\right>\geq (2n-1)k_2|p^h|^2$,
\end{enumerate}
for some non-negative constants $k_1$ and $k_2$, where $w_1,...,w_{2n-2}$ is an orthonormal frame of $\{p^h,\bJ p^h,v_0\}^\perp$. Then the first conjugate time of the geodesic $t\mapsto \psi_t(x,p)$ is less than or equal to $\frac{2\pi}{\sqrt{p(v_0)^2+k_1|p^h|^2}}$ and $\frac{2\pi}{\sqrt{p(v_0)^2+4k_2|p^h|^2}}$.

Moreover, if
\begin{enumerate}
\item $\left<\Rm^*(\bJ p^h,p^h)p^h,\bJ p^h\right>= k_1|p^h|^4$,
\item $\sum_{i=1}^{2n-2}\left<\Rm^*(w_i,p^h)p^h,w_i\right>= (2n-1)k_2|p^h|^2$.
\end{enumerate}
Then the first conjugate time of the geodesic $t\mapsto \psi_t(x,p)$ is equal to the minimum of $\frac{2\pi}{\sqrt{p(v_0)^2+k_1|p^h|^2}}$ and $\frac{2\pi}{\sqrt{p(v_0)^2+4k_2|p^h|^2}}$.
\end{thm}

\begin{proof}
Let $E(t)=(E^1(t),E^2(t),E^3(t)), F(t)=(F^1(t),F^2(t),F^3(t))$ be a canonical frame of the Jacobi curve $\mathfrak J_{(x,p)}(t)$. Let $A(t)$ and $B(t)$ be matrices defined by
\begin{equation}\label{E}
E(0)=A(t)E(t)+B(t)F(t).
\end{equation}

On the other hand, if we differentiate the equation (\ref{E}) with respect to $t$, then
\[
\begin{split}
0&=\dot A(t)E(t)+A(t)\dot E(t)+\dot B(t)F(t)+B(t)\dot F(t)\\
&=\dot A(t)E(t)+A(t)C_1E(t)+A(t)C_2F(t)\\
&+\dot B(t)F(t)-B(t)R(t)E(t)-B(t)C_1^TF(t).
\end{split}
\]

It follows that
\begin{equation}\label{AB}
\begin{split}
&\dot A(t)+A(t)C_1-B(t)R(t)=0\\
& \dot B(t)+A(t)C_2-B(t)C_1^T=0
\end{split}
\end{equation}
with initial conditions $B(0)=0$ and $A(0)=I$.

If we set $S(t)=B(t)^{-1}A(t)$, then $S(t)$ satisfies the following Riccati equation
\begin{equation}\label{RiccS}
\dot S(t)-S(t)C_2S(t)+C_1^TS(t)+S(t)C_1-R(t)=0.
\end{equation}

Let us choose $E^3_{2n-1}(0)=p^h+p(v_0)v$ and let
\[
S(t)=\left(\begin{array}{ccc}
S_1(t) & S_2(t) & S_3(t)\\
S_2(t)^T & S_4(t) & S_5(t)\\
S_3(t)^T & S_5(t)^T & S_6(t)
\end{array}\right),
\]
where $S_1(t)$ is a $2\times 2$ matrix and $S_6(t)$ is $1\times 1$. Then
\begin{equation}\label{Seqn}
\begin{split}
&\dot S_1(t)-S_1(t)\tilde C_2S_1(t)-S_2(t)S_2(t)^T\\
&-S_3(t)S_3(t)^T+\tilde C_1^TS_1(t)+S_1(t)\tilde C_1-\tilde R^1(t)=0,\\
&\dot S_4(t)-S_4(t)^2-S_5(t)S_5(t)^T-S_2(t)^T\tilde C_2S_2(t)-\tilde R^2(t)=0,\\
&\dot S_6(t)-S_6(t)^2-S_5(t)^TS_5(t)-S_3(t)^T\tilde C_2S_3(t)=0,
\end{split}
\end{equation}
where $\tilde C_1=\left(\begin{array}{cc}
0 & 1 \\
0 & 0 \\
\end{array}
\right)$,\ $\tilde C_2=\left(\begin{array}{cc}
0 & 0 \\
0 & 1 \\
\end{array}
\right)$, and $\tilde R^1(t)=\left(\begin{array}{cc}
0 & 0\\
0 & R^{22}(t)\\
 \end{array}
\right)$. $\tilde R^2(t)$ is the $(2n-2)\times(2n-2)$ matrix with $ij$-th entry equal to $R^{33}_{ij}(t)$.

Note that $U(t)=S(t)^{-1}$ also satisfies $U(0)=0$ and the Riccati equation
\[
\dot U(t)+C_2-U(t)C_1^T-C_1U(t)+U(t)R(t)U(t)=0.
\]
This gives
\[
U(t)=-tC_2-\frac{t^2}{2}(C_1+C_1^T)-\frac{t^3}{6}(C_1C_1^T+C_2R(0)C_2)+O(t^4).
\]
By using this expansion and $S(t)U(t)=I$, we obtain
\[
S_1(t)=\left(\begin{array}{cc}
-\frac{12}{t^3}+O(1/t^2) & \frac{6}{t^2}+O(1/t)\\
\frac{6}{t^2}+O(1/t) & -\frac{4}{t}+O(1)
\end{array}\right),
\]
\[
\tr(S_4(t))=-\frac{2n-2}{t}+O(1), \quad S_6(t)=-\frac{1}{t}+O(1).
\]
(For instance, one can take the dot product of the first row
\[
s(t)=(S_{1,1}(t),...,S_{1, 2n+1}(t))
\]
of $S(t)$ with the third, fourth, ..., $2n$-th columns of $U(t)$. This gives the order of the dominating terms of $(S_{1,3}(t),...,S_{1, 2n+1}(t))$ in terms of that of $S_{1,2}(t)$. By taking the dot product of $s(t)$ with the first and second column of $U(t)$, we obtain the leading order terms of $S_{1,1}(t)$ and $S_{1,2}(t)$. Similar procedure works for other entries of $S(t)$.)

By applying the comparison principle of Riccati equations in \cite{Ro} to $S(t)$, we have $S_1(t)\geq \Gamma_1(t)$, where $\Gamma_1(t)$ is a solution of the following Riccati equation
\[
\dot\Gamma_1(t)-\Gamma_1(t)\tilde C_2\Gamma_1(t)+\tilde C_1^T\Gamma_1(t)+\Gamma_1(t)\tilde C_1-K_1=0
\]
with the initial condition $\lim_{t\rightarrow 0}\Gamma_1^{-1}(t)=0$. (Of course, one needs to apply the comparison principle to $S(t)$ and $\Gamma(t+\epsilon)$ and let $\epsilon$ to zero as usual). Here $K_1=\left(\begin{array}{cc}
0 & 0\\
0 & \mathfrak k_1
\end{array}\right)$ and $\mathfrak k_1=p(v_0)^2+k_1|p^h|^2$.
Thus
\begin{equation}\label{split1}
\begin{split}
\tr(\tilde C_2S_1(t))&\geq\tr(\tilde C_2\Gamma_1(t))\\
&=\frac{\sqrt{\mathfrak k_1}(\sqrt{\mathfrak k_1}t\cos(\sqrt{\mathfrak k_1}t)-\sin(\sqrt{\mathfrak k_1}t))}{(2-2\cos(\sqrt{\mathfrak k_1}t)-\sqrt{\mathfrak k_1}t\sin(\sqrt{\mathfrak k_1}t))}.
\end{split}
\end{equation}

For the term $S_4(t)$, we can take the trace and obtain
\[
\begin{split}
\frac{d}{dt} \tr(S_4(t))&\geq \frac{1}{2n-2}\tr(S_4(t))^2+(2n-2)\mathfrak k_2,
\end{split}
\]
where $\mathfrak k_2=\frac{1}{4}p(v_0)^2+k_2|p^h|^2$.

Now applying the comparison principle  in \cite{Ro} again we have
\begin{equation}\label{split2}
\tr (S_4(t))\geq -\sqrt{\mathfrak k_2}(2n-2)\cot(\sqrt{\mathfrak k_2}t)).
\end{equation}

Finally, for the term $S_6(t)$, we  have
\[
\dot S_6(t)\geq S_6(t)^2.
\]
which implies
\[
S_6(t)\geq -\frac{1}{t}.
\]

By combining this with (\ref{split1}) and (\ref{split2}), we obtain
\begin{equation}\label{TraBdd}
\begin{split}
\tr(C_2S(t))&\geq -\sqrt{\mathfrak k_2}(2n-2)\cot\left(\sqrt{\mathfrak k_2}t\right)-\frac{1}{t}\\
&+\frac{\sqrt{\mathfrak k_1}(\sqrt{\mathfrak k_1}t\cos(\sqrt{\mathfrak k_1}t)-\sin(\sqrt{\mathfrak k_1}t))}{(2-2\cos(\sqrt{\mathfrak k_1}t)-\sqrt{\mathfrak k_1}t\sin(\sqrt{\mathfrak k_1}t))}.
\end{split}
\end{equation}

Therefore,
\[
\begin{split}
&\frac{d}{dt}\log|\det B(t)|=\text{tr}(C_1^T-S(t)C_2)=-\text{tr}(C_2S(t))\\
&\leq \sqrt{\mathfrak k_2}(2n-2)\cot\left(\sqrt{\mathfrak k_2}t\right)+\frac{1}{t}-\frac{\sqrt{\mathfrak k_1}(\sqrt{\mathfrak k_1}t\cos(\sqrt{\mathfrak k_1}t)-\sin(\sqrt{\mathfrak k_1}t))}{(2-2\cos(\sqrt{\mathfrak k_1}t)-\sqrt{\mathfrak k_1}t\sin(\sqrt{\mathfrak k_1}t))}
\end{split}
\]
and hence
\[
\begin{split}
&|\det B(t)|\leq C a(t)\\
\end{split}
\]
where $C=\lim_{t_0\to 0}\frac{|\det B(t_0)|}{a(t_0)}$ and
\[
a(t)=t\sin^{2n-2}(\sqrt{\mathfrak k_2}t)(2-2\cos(\sqrt{\mathfrak k_1} t)-\sqrt{\mathfrak k_1}t\sin(\sqrt{\mathfrak k_1} t)).
\]

Using (\ref{AB}) and the definition of determinant, we see that $B(t)=-C_2t+\frac{1}{2}(C_1-C_1^T)t^2+\frac{1}{6}(C_2R(0)C_2+C_1C_1^T)t^3+O(t^4)$ and $|\det B(t)|=\frac{1}{12}t^{2n+3}+O(t^{2n+4})$.

Therefore,
\[
\begin{split}
&|\det B(t)|\leq \frac{t\sin^{2n-2}(\sqrt{\mathfrak k_2}t)(2-2\cos(\sqrt{\mathfrak k_1} t)-\sqrt{\mathfrak k_1}t\sin(\sqrt{\mathfrak k_1} t))}{\mathfrak k_1^2\mathfrak k_2^{2n-2}}.
\end{split}
\]

The first assertion follows.
Let $S^{k_1,k_2}(t)$ be a solution of (\ref{RiccS}) with $R(t)$ replaced by
\[
R^{k_1,k_2}=\left(\begin{array}{cccc}
0 & 0 & 0 & 0\\
0 & \mathfrak k_1 & 0 & 0\\
0 & 0 & \mathfrak k_2I_{2n-2} & 0\\
0 & 0 & 0 & 0
\end{array}\right)
\]
with the initial condition $\lim_{t\rightarrow 0}(S_t^{k_1,k_2})^{-1}=0$, where $\mathfrak k_1=p(v_0)^2$ and $\mathfrak k_2=\frac{1}{4}p(v_0)^2$.

A calculation similar to that of Theorem \ref{conj} shows that
\[
S^{k_1,k_2}(t)=\left(
\begin{array}{cccc}
 \frac{-(k_1)^{3/2}\sin(\tau_t)}{s(t)}& \frac{k_1(1-\cos(\tau_t))}{s(t)} & 0 & 0 \\
\frac{k_1(1-\cos(\tau_t))}{s(t)} &   \frac{(k_1)^{1/2}(\tau_t\cos(\tau_t)-\sin(\tau_t))}{s(t)} & 0 &0\\
0 & 0 & -\sqrt{\mathfrak k_2}\cot(\sqrt{\mathfrak k_2} t)I_{2n-2} & 0\\
0 & 0&0 & -\frac{1}{t}
\end{array}\right),
\]
where $\tau_t=\sqrt{\mathfrak k_1}t$ and $s(t)=2-2\cos(\tau_t)-\tau_t\sin(\tau_t)$.

The rest follows as the proof of the previous assertion (with all inequalities replaced by equalities).
\end{proof}

\smallskip

\section{Model Cases}\label{Model}

In this section, we discuss two examples, the Heisenberg group and the complex Hopf fibration which are relevant to the later sections. First, we consider a family of Sasakian manifolds $(M,\bJ,v_0,\alpha_0,g=\left<\cdot,\cdot\right>)$ for which the quotient of $M$ by the flow of $v_0$ is a manifold $B$. Since $\mathcal L_{v_0}\bJ=0$ and $\mathcal L_{v_0}g=0$, they descend to a complex structure $\bJ_B$ and a Riemannian metric $g_B$ on $B$. Moreover, by Theorem \ref{Sasaki2}, they form a K\"ahler manifold. Moreover, the Tanaka-Webster curvature $\Rm^*$ on $M$ and the Riemann curvature tensor $\Rm^B$ of $B$ are related by

\begin{lem}\label{relcurv}
The curvature tensors $\Rm^*$ and $\Rm^B$ are related by
\[
\Rm^*(\bar X,\bar Y)\bar Z=\overline{\Rm^B(X,Y)Z},
\]
where $\bar X$ denotes the vector orthogonal to $v_0$ which project to the vector $X$.
\end{lem}

\begin{proof}
Since $M\to B$ is a Riemannian submersion, we have (see \cite{On})
\[
\nabla_{\bar X}^*\bar Y=(\nabla_{\bar X}\bar Y)^h=\overline{\nabla_XY}.
\]

Since $\bar Z$ projects to $Z$, we also have
\[
\nabla^*_{v_0}\bar Z=(\nabla_{v_0}\bar Z)^h+\frac{1}{2}\bJ \bar Z=(\nabla_{\bar Z}v_0)^h+\frac{1}{2}\bJ \bar Z=0.
\]

Therefore,
\[
\begin{split}
&\Rm^*(\bar X,\bar Y)\bar Z=\nabla_{\bar X}^*\nabla_{\bar Y}^*\bar Z-\nabla_{\bar Y}^*\nabla_{\bar X}^*\bar Z-\nabla_{[\bar X,\bar Y]}^*\bar Z\\
&=\overline{\nabla_{X}\nabla_{Y}Z}-\overline{\nabla_{Y}\nabla_{X}Z}-\overline{\nabla_{[X,Y]}Z}-\alpha_0([\bar X,\bar Y])\nabla_{v_0}^*\bar Z\\
&=\overline{\Rm^B(X,Y)Z}.
\end{split}
\]

\end{proof}

The first example is the Heisenberg group. In this case the manifold $M$ is the Euclidean space $\Real^{2n+1}$. If we fix a coordinate system $(x_1,...,x_n,y_1,...,y_n,z)$, then the 1-form $\alpha_0$ and the vector field $v_0$, are given, respectively, by
\[
\alpha_0=dz-\frac{1}{2}\sum_{i=1}^nx_idy_i+\frac{1}{2}\sum_{i=1}^ny_idx_i
\quad \text{ and } \quad v_0=\partial_z.
\]

The Riemannian metric is the one for which the frame
\[
X_i=\partial_{x_i}-\frac{1}{2}y_i\partial_z,\quad Y_i=\partial_{y_i}+\frac{1}{2}x_i\partial_z,\quad \partial_z
\]
is orthonormal. The tensor $\bJ$ is defined by
\[
\bJ(X_i)=Y_i,\quad \bJ(Y_i)=-X_i,\quad \bJ(\partial_z)=0.
\]
The quotient $B$ is $\Complex^n$ equipped with the standard complex structure and Euclidean inner product.

Let $(x,p)$ be a covector with $|p^h|=1$. Assume that $t\mapsto\psi(x,t\epsilon p)$ is length minimizing between its endpoints for some $\epsilon>0$. Then, we define the cut time of $(x,p)$ to be the largest such $\epsilon$. The following is well-known. We give the proof for completeness.

\begin{thm}\label{Hei}
On the Heisenberg group equipped with the above sub-Riemannian structure, the cut time coincides with the first conjugate time.
\end{thm}

\begin{proof}
Let $P_{X_i}=p_{x_i}-\frac{1}{2}y_ip_z$ and $P_{Y_i}=p_{y_i}+\frac{1}{2}x_ip_z$. A computation as in \cite{Mo} shows that
\[
\begin{split}
&P_j(t):=P_{X_j}(t)+iP_{Y_j}(t)=P_j(0)e^{itp_z},\\
&w_j(t):=x_j(t)+iy_j(t)=w_j(0)-\frac{iP_j(0)}{p_z}(e^{itp_z}-1),\\
&z(t):=z(0)+\frac{1}{2}\sum_{k=1}^n\int_0^t\text{Im}(\bar w_k(s)\dot w_k(s))ds.
\end{split}
\]

If $(w, z)$ and $(\tilde w,\tilde z)$ are unit speed geodesics with the same length $L$ and end-points, then
\[
\frac{\tilde P_j(0)}{\tilde p_z}(e^{iL\tilde p_z}-1)=\frac{P_j(0)}{p_z}(e^{iLp_z}-1).
\]
By taking the norms, it follows that
\[
\frac{1-\cos(L\tilde p_z)}{\tilde p_z^2}=\frac{1-\cos(Lp_z)}{p_z^2}.
\]

Using $w_j(0)=\tilde w_j(0)$ and $w_j(L)=\tilde w_j(L)$, we also have
\[
\frac{e^{i\tilde\theta}}{\tilde p_z}(e^{iL\tilde p_z}-1)=\frac{e^{i\theta}}{p_z}(e^{iLp_z}-1),
\]
where $P_j(0)=e^{i\theta}$ and $\tilde P_j(0)=e^{i\tilde \theta}$. Therefore,
\[
\begin{split}
&\frac{\cos(\theta+Lp_z)-\cos(\theta)}{p_z}=\frac{\cos(\tilde\theta+L\tilde p_z)-\cos(\tilde\theta)}{\tilde p_z},\\
&\frac{\sin(\theta+Lp_z)-\sin(\theta)}{p_z}=\frac{\sin(\tilde \theta+L\tilde p_z)-\sin(\tilde \theta)}{\tilde p_z}.
\end{split}
\]

Finally, since $z(L)=\tilde z(L)$, a computation together with the above implies that
\[
\frac{L\tilde p_z-\sin(L\tilde p_z)}{\tilde p_z^2}=\frac{Lp_z-\sin(Lp_z)}{p_z^2}.
\]

By investigating the graph of $\frac{1-\cos(x)}{x^2}$ and $\frac{x-\sin(x)}{x^2}$, we have $p_z=\tilde p_z$. Therefore, if $L<\frac{2\pi}{p_z}$, then $P_j(0)=\tilde P_j(0)$ and the two geodesics coincide. Hence, the result follows from Theorem \ref{conj}.
\end{proof}

The second example is the complex Hopf fibration. We follow the discussion in \cite{ChMaVa}. In this case, the manifold is given by the sphere $S^{2n+1}=\{z\in\Complex^{n+1}| |z|=1\}$. The 1-form $\alpha_0$ and the vector field $v_0$ are given, respectively, by
\[
\alpha_0=\frac{1}{2}\sum_{i=1}^n(x_idy_i-y_idx_i)
\]
and
\[
v_0=2\sum_{i=1}^n\left(-y_i\partial_{x_i}+x_{i}\partial_{y_i}\right)
\]
where $z_j=x_j+iy_j$.

The tangent space of $S^{2n+1}$ is the direct sum of $\ker\alpha_0$ and $\Real v_0$. The Riemannian metric is defined in such a way that $v_0$ has length one, $v_0$ is orthogonal to $\ker\alpha_0$, and the restriction of the metric to $\ker\alpha_0$ coincides with the Euclidean one. The (1,1)-tensor $\bJ$ is defined analogously by the conditions $\bJ v_0=0$ and the restriction of $\bJ$ to $\ker\alpha_0$ coincides with the standard complex structure on $\Complex^n$. The base manifold $B$ is the complex projective space $\mathbb{CP}^n$ and the induced Riemannian metric is given by the Fubini-Study metric. It follows from Lemma \ref{relcurv} that
\[
\left<\Rm^*(\bJ X,X)X,\bJ X\right>=4 \text{ and } \left<\Rm^*(v,X)X,v\right>=1
\]
for all $v$ in the orthogonal complement of $\{X,JX\}$.

\begin{thm}\label{hopf}
On the complex Hopf fibration equipped with the above sub-Riemannian structure, the cut time coincides with the first conjugate time.
\end{thm}

\begin{proof}
The sub-Riemannian geodesic flow is given by
\[
\left(a\cos(|v|t)+\frac{v}{|v|}\sin(|v|t)\right)e^{-it\left<v_0,v\right>},
\]
where $a$ is the initial point of the geodesic and $v$ is the initial (co)vector (see \cite{Mo,ChMaVa}).

By the choice of the complex coordinate system, we can assume $a=(1,0,...,0)$. Let $v=(v_1,...,v_n)$. Then the real part of $v_1$ equal $0$. Moreover, $v^h=(0,v_2,...,v_n)$ is the horizontal part of $v$. Assume that $|v^h|=1$ and let $w$ be another such covector such that the corresponding geodesic has the same end point and the same length $L$ as that of $v$.

Under the above assumptions, we have
\[
|v|^2-\frac{1}{4}(\text{Im}(v_1))^2=1=|w|^2-\frac{1}{4}(\text{Im}(w_1))^2
\]
and
\[
\begin{split}
&\left(a\cos(||v||L)+\frac{v}{||v||}\sin(||v||L)\right)e^{-\frac{iL}{2}\text{Im}(v_1)} \\ &=\left(a\cos(||w||L)+\frac{w}{||w||}\sin(||w||L)\right)e^{-\frac{iL}{2}\text{Im}(w_1)}.
\end{split}
\]

It follows that
\[
\begin{split}
&\left(\cos(|v|L)+\frac{v_1}{|v|}\sin(|v|L)\right)e^{-\frac{iL}{2}\text{Im}(v_1)} \\ &=\left(\cos(|w|L)+\frac{w_1}{|w|}\sin(|w|L)\right)e^{-\frac{iL}{2}\text{Im}(w_1)}
\end{split}
\]
and
\[
\begin{split}
&\left(\frac{v_i}{|v|}\sin(|v|L)\right)e^{-\frac{iL}{2}\text{Im}(v_1)} =\left(\frac{w_i}{|w|}\sin(|w|L)\right)e^{-\frac{iL}{2}\text{Im}(w_1)}.
\end{split}
\]
for all $i\neq 1$.

By taking the norm of the second equation, we obtain
\[
\begin{split}
&\frac{|v_i|^2}{|v|^2}\sin^2(|v|L)=\frac{|w_i|^2}{|w|^2}\sin^2(|w|L).
\end{split}
\]

If we sum over $i\neq 1$, then we have
\[
\begin{split}
&\frac{\sin^2(|v|L)}{|v|^2}=\frac{\sin^2(|w|L)}{|w|^2}.
\end{split}
\]

If both $|v|$ and $|w|$ are less than or equal to $\frac{\pi}{L}$, then $|v|=|w|$. It follows that $\text{Im}(v_1)=\pm\text{Im}(w_1)$.

If $\text{Im}(v_1)=\text{Im}(w_1)$, then either $v_i=w_i$ for all $i$ which implies that the two geodesics coincide or $\sin(L|v|)=0=\sin(L|w|)$. In this case $|v|=|w|=\frac{\pi}{L}$.

If $\text{Im}(v_1)=-\text{Im}(w_1)$, then
\[
\begin{split}
&\left(\cos(|v|)+\frac{v_1}{|v|}\sin(|v|)\right)e^{iL\text{Im}(v_1)} =\left(\cos(|v|)-\frac{v_1}{|v|}\sin(|v|)\right).
\end{split}
\]
It follows that
\[
\frac{\tan(|v|)}{|v|}=\frac{\tan(\text{Im}(v)/2)}{\text{Im}(v)/2}.
\]
Since $|v|>\frac{1}{2}\text{Im}(v)$, we have a contradiction. Therefore, the result from this and Theorem \ref{conj}.
\end{proof}

\smallskip

\section{Volume growth estimates}\label{volestimate}

In this section, we prove a volume growth estimate and the proof of Theorem \ref{main1-1} and \ref{main1-2}. Let $\Omega$ be the set of points $(x,p)$ in the cotangent space $T^*_xM$ such that the curve $t\in [0,1]\mapsto \psi_t(x,p)$ is a length minimizing. Let
\[
\Sigma=\{p\in \Omega| |p^h|=1\text{ and } \epsilon p\in\Omega \text{ for some } \epsilon>0\}.
\]
For each $p$ in $\Sigma$, we let $T(p)$ be the cut time which is the maximal time $T$ such that $t\in [0,T]\mapsto \psi_t(x,p)$ is length minimizing. Finally, let us denote the ball centered at $x$ of radius $R$ with respect to the sub-Riemannian distance by $B_R(x)$ and the Riemannian volume form by $\eta$.

\begin{thm}\label{volgrowth}
Assume that the Tanaka-Webster curvature $\Rm^*$ of the Sasakian manifold satisfies
\begin{enumerate}
\item $\left<\Rm^*(\bJ p^h,p^h)p^h,\bJ p^h\right>\geq k_1|p^h|^4$,
\item $\sum_{i=1}^{2n-2}\left<\Rm^*(w_i,p^h)p^h,w_i\right>\geq (2n-1)k_2|p^h|^2$,
\end{enumerate}
for some constants $k_1$ and $k_2$, where $w_1,...,w_{2n-2}$ is an orthonormal frame of $\textbf{span}\{p^h,\bJ p^h,v_0\}^\perp$. Then
\[
\int_{B_R(x)}d\eta\leq \int_0^{\min\{T(p),R\}}\int_{\Sigma} k(r,z) d\mathfrak m(r,z)
\]
where $(r,z)$ denotes the cylindrical coordinates defined by $r=|p^h|$ and $z=p(v_0)$, $\mathfrak k_1(r,z)=z^2+k_1r^2$, $\mathfrak k_2(r,z)=\frac{1}{4}z^2+k_2r^2$. The function $k$ is defined by
\[
\begin{split}
&k(r,z)=r^2\left[\frac{\sin^{2n-2}(\sqrt{\mathfrak k_2})(2-2\cos(\sqrt{\mathfrak k_1})-\sqrt{\mathfrak k_1}\sin(\sqrt{\mathfrak k_1}))}{\mathfrak k_1^2\mathfrak k_2^{2n-2}}\right]
\end{split}
\]
if $\mathfrak k_1\geq 0$ and $\mathfrak k_2\geq 0$,
\[
\begin{split}
&k(r,z)=r^2\left[\frac{\sinh^{2n-2}(\sqrt{-\mathfrak k_2})(2-2\cos(\sqrt{\mathfrak k_1})-\sqrt{\mathfrak k_1}\sin(\sqrt{\mathfrak k_1}))}{\mathfrak k_1^2\mathfrak k_2^{2n-2}}\right]
\end{split}
\]
if $\mathfrak k_1\geq 0$ and $\mathfrak k_2\leq 0$,
\[
\begin{split}
&k(r,z)=r^2\left[\frac{\sin^{2n-2}(\sqrt{\mathfrak k_2})(2-2\cosh(\sqrt{-\mathfrak k_1})+\sqrt{-\mathfrak k_1}\sinh(\sqrt{-\mathfrak k_1}))}{\mathfrak k_1^2\mathfrak k_2^{2n-2}}\right]
\end{split}
\]
if $\mathfrak k_1\leq 0$ and $\mathfrak k_2\geq 0$,
\[
\begin{split}
&k(r,z)=r^2\left[\frac{\sinh^{2n-2}(\sqrt{-\mathfrak k_2})(2-2\cosh(\sqrt{-\mathfrak k_1})+\sqrt{-\mathfrak k_1}\sinh(\sqrt{-\mathfrak k_1}))}{\mathfrak k_1^2\mathfrak k_2^{2n-2}}\right]
\end{split}
\]
if $\mathfrak k_1\leq 0$ and $\mathfrak k_2\leq 0$.
\end{thm}

\begin{proof}
We use the same notations as in the proof of Theorem \ref{conj}.

Let $\rho_t:T^*_xM\to\Real$ be the function defined by $\psi_t^*\eta=\rho_t\mathfrak m$. It follows from Theorem \ref{vol} that
\begin{equation}\label{rho}
\rho_t=|p^h|^2|\det B(t)|.
\end{equation}

Next, we replace the matrix $R(t)$ in (\ref{AB}) by $R^{k_1,k_2}$ and denote the solutions by $A^{k_1,k_2}(t)$ and $B^{k_1,k_2}(t)$. Then
\[
\frac{\frac{d}{dt}\det B(t)}{\det B(t)}=-\text{tr}(S(t)C_2)\leq -\text{tr}(S^{k_1,k_2}(t)C_2)=\frac{\frac{d}{dt}\det B^{k_1,k_2}(t)}{\det B^{k_1,k_2}(t)}.
\]
It follows that $\frac{\det B(t)}{\det B^{k_1,k_2}(t)}$ is non-increasing.

It follows that from the proof of Theorem \ref{conj} that
\[
\begin{split}
&\int_{B_R(x)}d\eta = \int_\Sigma\int_0^{\text{min}\{T(p),R\}}\rho_t d\mathfrak m\\
&\leq \int_{\Omega_R}|p^h|^2\left[\frac{\sin^{2n-2}(\sqrt{\mathfrak k_2})(2-2\cos(\sqrt{\mathfrak k_1})-\sqrt{\mathfrak k_1}\sin(\sqrt{\mathfrak k_1}))}{\mathfrak k_1^2\mathfrak k_2^{2n-2}}\right]d\mathfrak m(p).
\end{split}
\]
\end{proof}

\begin{proof}[Proof of Theorem \ref{main1-1} and \ref{main1-2}]
By the proof of Theorem \ref{conj} and Theorem \ref{hopf}, the volume of sub-Riemannian ball of radius $R$ in the Complex Hopf fibration is given by
\[
\begin{split}
&\int_{\Omega_R}|p^h|^2\left[\frac{\sin^{2n-2}(\sqrt{\mathfrak k_2})(2-2\cos(\sqrt{\mathfrak k_1})-\sqrt{\mathfrak k_1}\sin(\sqrt{\mathfrak k_1}))}{\mathfrak k_1^2\mathfrak k_2^{2n-2}}\right]d\mathfrak m(p).
\end{split}
\]

Therefore, the result follows from \ref{volgrowth}.
\end{proof}

\smallskip

\section{Laplacian comparison theorem}\label{Laplace}

In this section, we define a version of Hessian following \cite{AgLe2} and prove Theorem \ref{main2}.

Let $f:M\to\Real$ be a smooth function. The graph $G$ of the differential $df$ defines a sub-manifold of the manifold $T^*M$. Let $v$ be a tangent vector in $T_xM$. Then there is a vector $X$ in the tangent space of $G$ at $df_x$ such that $\pi_*(X)=v$, where $\pi:T^*M\to M$ is the projection. The sub-Riemannian Hessian $\hess\,f$ at $x$ is defined by $\hess\, f(v)=X_\ver$. Recall that $X_\ver$ is the component of $X$ in $\ver$ with respect to the splitting $TT^*M=\ver\oplus\hor$.

\begin{lem}
Under the identification in Theorem \ref{splitcan}, the sub-Riemannian Hessian is given by
\begin{enumerate}
\item $\hess\, f(v)=\nabla_v\nabla f$ if $v$ is contained in the orthogonal complement of $\{\nabla f^h, \bJ\nabla f, v_0\}$,
\item $\hess\, f(\nabla f^h)=\nabla_{\nabla f^h}\nabla f-\frac{1}{2}\left<\nabla f,v_0\right>\bJ \nabla f^h$,
\item $\hess\, f(\bJ\nabla f)=\nabla_{\bJ\nabla f}\nabla f-\frac{1}{2}\left<\nabla f,v_0\right>\nabla f^h+\frac{1}{2}|\nabla f^h|^2v_0$,
\item $\hess\, f(v)=\nabla_{v}\nabla f+\frac{|\nabla f|^2}{2}\bJ\nabla f$ if $v=|\nabla f^h|^2v_0-(v_0f)\nabla f^h$.
\end{enumerate}
\end{lem}

\begin{proof}
Let $\{v_0,...,v_{2n}\}$ be a frame defined as in Lemma \ref{Chris} around a point $x$. Since $\pi_*(\vec h_i)=v_i$, we have
\[
(df)_*(k_av_a)=k_a\vec h_a+\bar k_a\vec\alpha_a.
\]
It follows that
\[
\begin{split}
&\bar k_c+k_adh_a(\vec h_c)=\omega(\vec h_c,(df)_*(k_av_a))=-dh_c((df)_*(k_av_a))\\
&=-k_a(v_av_cf)=-k_a\left<\nabla_{v_a}\nabla f,v_c\right>-k_a\left<\nabla f,\nabla_{v_a}v_c\right>.
\end{split}
\]

Therefore, we have the following at $x$.
\[
\begin{split}
\bar k_i&=-k_a\left<\nabla_{v_a}\nabla f,v_i\right>-k_a\left<\nabla f,\nabla_{v_a}v_i\right>-k_adh_a(\vec h_i)\\
&=-k_a\left<\nabla_{v_a}\nabla f,v_i\right>-\frac{k_j}{2}\bJ_{ji}v_0f-k_0dh_0(\vec h_i)-k_jdh_j(\vec h_i)\\
&=-k_a\left<\nabla_{v_a}\nabla f,v_i\right>+\frac{k_j}{2}\bJ_{ji}v_0f+\frac{k_0}{2}\bJ_{ik}v_k f
\end{split}
\]
and
\[
\begin{split}
&\bar k_0=-k_a\left<\nabla_{v_a}\nabla f,v_0\right>-k_i\left<\nabla f,\nabla_{v_i}v_0\right>-k_idh_i(\vec h_0)\\
&=-k_a\left<\nabla_{v_a}\nabla f,v_0\right>+\frac{k_i}{2}\left<\bJ v_i,\nabla f^h\right>-\frac{1}{2}k_i\bJ_{ij}h_j=-k_a\left<\nabla_{v_a}\nabla f,v_0\right>.
\end{split}
\]

Hence, if $v:=k_av_a$ is contained in $\pi_*\hor_3$, then
\[
((df)_*(k_iv_i))_\ver=-\left(\frac{1}{2}k_j\bJ_{ji} v_0f-\frac{(v_0f)(v_s f)k_s}{2|\nabla f^h|^2}(v_jf)\bJ_{ji}\right)\vec \alpha_i+\bar k_a\vec\alpha_a.
\]

If $v$ is contained in $\pi_*\hor_3$ and the orthogonal complement of $\nabla f^h$, then
\[
\begin{split}
&((df)_*(k_iv_i))_\ver=-\left<\nabla_{k_iv_i}\nabla f,v_a\right>\vec\alpha_a.
\end{split}
\]

If $v=\nabla f^h$, then
\[
((df)_*(\nabla f^h))_\ver=-\left<\nabla_{\nabla f^h}\nabla f,v_a\right>\vec\alpha_a+\frac{1}{2}\left<\bJ \nabla f^h,v_i\right>\left<\nabla f,v_0\right>\vec\alpha_i.
\]

If $v=\bJ\nabla f^h$, then
\[
\begin{split}
((df)_*(\bJ\nabla f))_\ver&=[(v_jf)\bJ_{ji}\vec h_i]_\ver-\left<\nabla_{\bJ\nabla f}\nabla f,v_0\right>\vec\alpha_0\\
&-\left<\nabla_{\bJ\nabla f}\nabla f,v_i\right>\vec\alpha_i-\frac{v_if}{2}(v_0f)\vec\alpha_i\\
&=-\left<\nabla_{\bJ\nabla f}\nabla f,v_a\right>\vec\alpha_a+\frac{v_if}{2}(v_0f)\vec\alpha_i-\frac{1}{2}|\nabla f^h|^2\vec\alpha_0.
\end{split}
\]

Finally, if $v=|\nabla f^h|^2v_0-(v_0f)\nabla f^h$, then we have
\[
((df)_*(v))_\ver=-\left<\nabla_{v}\nabla f,v_a\right>\vec\alpha_a-\frac{|\nabla f|^2}{2}\left<\bJ\nabla f,v_i\right>\vec\alpha_i.
\]
\end{proof}

\begin{proof}[Proof of Theorem \ref{main2}]
Let $f(x)=-\frac{1}{2}d^2(x,x_0)$. Then the curve $t\in[0,1]\mapsto \pi e^{t\vec H}(df_x)$ is the geodesic which starts from $x$ and ends at $x_0$. Let $E(t)=(E^1(t),E^2(t),E^3(t)), F(t)=(F^1(t),F^2(t),F^3(t))$ be a canonical frame of the Jacobi curve $\mathfrak J_{(x,df_x)}(t)$. Let
\[
\E=(\E^1,\E^2,\E^3_1,...,\E^3_{2n-1})^T,\cF=(\cF^1,\cF^2,\cF^3_1,...,\cF^3_{2n-1})^T
\]
be a symplectic basis of $T_{(x_0,p)}T^*M$ such that $\E^i$ is contained in $\ver_i$ and $\cF^i$ is contained in $\hor_i$, where $(x_0,p)=e^{1\cdot\vec H}(df_x)$. Let
\[
v=(v^1,v^2,v^3_1,...,v^3_{2n-1})^T
\]
be a basis of $T_xM$ such that $e^{t\vec H}_*(df_x)_*(v)=\E$. Let $A(t)$ and $B(t)$ be matrices such that
\[
(df_x)_*(v)=A(t)E(t)+B(t)F(t).
\]
By construction, we have $B(1)=0$. We can also pick $E(t)$ such that $A(1)=I$.

By the definition of $\hess\, f$, we also have
\[
\hess\, f(B(0)\pi_*F(0))=\hess\, f(v)=A(0)E(0).
\]
Therefore, if we let $S(t)=B(t)^{-1}A(t)$, then
\[
\hess\, f(\pi_*F(0))=S(0)E(0).
\]

A computation as in the proof of Theorem \ref{conj} shows that
\[
\dot S(t)-S(t)C_2S(t)+C_1^TS(t)+S(t)C_1-R(t)=0.
\]
Therefore, by applying similar computation as in the proof of Theorem \ref{conj} to $S(1-t)$, we obtain estimates for $S(0)$. Since $\Delta_H f(x)=\tr(C_2S(0))$, the result follows.
\end{proof}

\smallskip

\section{Appendix I}

In this section, we give the proof of various known results in Section 3.

\begin{proof}[Proof of Lemma \ref{Sasaki1}]
Since the almost contact manifold is normal, we have
\[
0=[\bJ,\bJ](v,v_0)+d\alpha_0(v,v_0)=\bJ^2[v,v_0]-\bJ[\bJ v,v_0]=\bJ\mathcal L_{v_0}(\bJ) v.
\]
It follows that $\mathcal L_{v_0}(\bJ)=0$.

Since the metric is associated to the almost contact structure,
\[
\begin{split}
0&=\LD_{v_0}\alpha_0(v)=\LD_{v_0}(\left<v_0,v\right>)-\alpha_0([v_0,v])\\
&=\left<\nabla_{v_0}v_0,v\right>+\left<v_0,\nabla_{v_0}v\right>-\left<v_0,\nabla_{v_0}v\right>+\left<v_0,\nabla_vv_0\right>\\
&=\left<\nabla_{v_0}v_0,v\right>.
\end{split}
\]

Since the metric is associated to the almost contact structure and $\mathcal L_{v_0}(\bJ)=0$, we also have
\[
\LD_{v_0}g(v,\bJ w)=(\LD_{v_o}d\alpha_0)(v,w)=0.
\]
Therefore, $\LD_{v_0}g=0$ as claimed.

By Lemma \ref{Chris}, we have
\[
\left<\bJ(v_j),v_i\right>=\bJ_{ji}=2\Gamma_{ji}^0=-2\left<\nabla_{v_j}v_0,v_i\right>.
\]

Therefore, $\bJ=-2\nabla v_0$.

\end{proof}

\begin{proof}[Proof of Lemma \ref{Sasaki2}]
Let $v_0,v_1,...,v_{2n}$ be a local frame defined by Lemma \ref{parallel2}. Then
\[
\begin{split}
0&=\mathcal L_{v_0}(\bJ)(v_i)=[v_0,\bJ v_i]-\bJ[v_0,v_i]\\
&=\nabla_{v_0}(\bJ v_i)-\nabla_{\bJ v_i}(v_0)-\bJ\nabla_{v_0}v_i+\bJ\nabla_{v_i}v_0\\
&=(\nabla_{v_0}\bJ) v_i-\nabla_{\bJ v_i}(v_0)+\bJ\nabla_{v_i}v_0\\
&=(\nabla_{v_0}\bJ) v_i+\frac{1}{2}\bJ^2v_i-\frac{1}{2}\bJ^2v_i=(\nabla_{v_0}\bJ) v_i\\
\end{split}
\]

Since $\bJ v_0=0$,
\[
(\nabla_{v_i}\bJ)v_0=-\bJ\nabla_{v_i}v_0=\frac{1}{2}\bJ\bJ v_i=-\frac{1}{2}v_i.
\]

Since $\nabla_{v_0}v_0=0$, we also have $(\nabla_{v_0}\bJ)v_0=-\bJ(\nabla_{v_0}v_0)=0$.

Finally, we need to show $(\nabla_{v_i}\bJ)v_j=\frac{1}{2}\delta_{ij}v_0$. First, by the properties of the frame $v_1,...,v_n$, we have
\[
\left<(\nabla_{v_i}\bJ)v_j,v_0\right>=-\left<\bJ v_j,\nabla_{v_i} v_0\right>=\frac{1}{2}\left<\bJ v_j,\bJ v_i\right>=\frac{1}{2}\delta_{ij}
\]
at $x$

By normality and properties of the frame $v_1,...,v_{2n}$, we have
\[
0=(\nabla_{\bJ v_i}\bJ) v_j-(\nabla_{\bJ v_j}\bJ) v_i+\bJ(\nabla_{v_j}\bJ) v_i-\bJ(\nabla_{v_i}\bJ) v_j+d\alpha_0(v_i,v_j)v_0.
\]
It follows from Lemma \ref{Chris} that
\[
\begin{split}
0&=\left<(\nabla_{\bJ v_i}\bJ) v_j,v_k\right>-\left<(\nabla_{\bJ v_j}\bJ) v_i,v_k\right>+\left<\bJ(\nabla_{v_j}\bJ) v_i,v_k\right>-\left<\bJ(\nabla_{v_i}\bJ) v_j,v_k\right>\\
&=-\left<(\nabla_{v_k}\bJ) \bJ v_i,v_j\right>-\left<(\nabla_{v_j}\bJ) v_k,\bJ v_i\right>+\left<(\nabla_{v_k}\bJ) \bJ v_j,v_i\right>\\
&+\left<(\nabla_{v_i}\bJ) v_k,\bJ v_j\right>+\left<\bJ(\nabla_{v_j}\bJ) v_i,v_k\right>-\left<\bJ(\nabla_{v_i}\bJ) v_j,v_k\right>\\
&=-\left<(\nabla_{v_k}\bJ) \bJ v_i,v_j\right>+\left<(\nabla_{v_j}\bJ) \bJ v_i,v_k\right>+\left<\bJ(\nabla_{v_k}\bJ)v_i, v_j\right>\\
&-\left<(\nabla_{v_i}\bJ)\bJ v_j, v_k\right>+\left<\bJ(\nabla_{v_j}\bJ) v_i,v_k\right>-\left<\bJ(\nabla_{v_i}\bJ) v_j,v_k\right>.
\end{split}
\]

Since $\bJ^2v_j=-v_j$, we also have $\left<(\nabla_{v_i}\bJ)\bJ v_j,v_k\right>=-\left<\bJ(\nabla_{v_i}\bJ) v_j,v_k\right>$. Therefore, the above equation simplifies to
\[
\begin{split}
0&=-2\left<(\nabla_{v_k}\bJ)v_i,\bJ v_j\right>.
\end{split}
\]
\end{proof}

\begin{proof}[Proof of \ref{Sasaki3}]
Since the manifold is Sasakian, we have
\[
\begin{split}
\Rm(X,Y)v_0&=\nabla_X\nabla_Yv_0-\nabla_Y\nabla_Xv_0-\nabla_{[X,Y]}v_0\\
&=\frac{1}{2}(-\nabla_X(\bJ(Y))+\nabla_Y(\bJ(X))+\bJ[X,Y])\\
&=\frac{1}{2}(-\nabla_X\bJ(Y)+\nabla_Y\bJ(X))\\
&=\frac{1}{4}\alpha_0(Y)X-\frac{1}{4}\alpha_0(X)Y.
\end{split}
\]
\end{proof}

\begin{proof}[Proof of \ref{Sasaki4}]
Let $\nabla^*$ be the Tanaka connection defined by
\[
\nabla^*_XY=\nabla_XY+\alpha_0(X)\bJ Y-\alpha_0(Y)\nabla_Xv_0+\nabla_X\alpha_0(Y)v_0
\]

Assume that $X$ and $Y$ are horizontal. Then
\[
\nabla^*_XY=\nabla_XY-\left<\nabla_XY,v_0\right>v_0.
\]
Therefore,
\[
\begin{split}
&\nabla^*_X\nabla^*_YZ=\nabla_X(\nabla_YZ-\left<\nabla_YZ,v_0\right>v_0) -\left<\nabla_X(\nabla_YZ-\left<\nabla_YZ,v_0\right>v_0),v_0\right>v_0\\
&=\nabla_X\nabla_YZ-\left<\nabla_X\nabla_YZ,v_0\right>v_0  -\left<\nabla_YZ,v_0\right>\nabla_Xv_0
\end{split}
\]

Let $\Rm^*$ be the curvature corresponding to $\nabla^*$. Assume that $X,Y,Z$ are horizontal. Then
\[
\begin{split}
&\Rm^*(X,Y)Z=\nabla^*_X\nabla^*_YZ-\nabla^*_Y\nabla^*_XZ-\nabla^*_{[X,Y]}Z\\
&=\nabla_X\nabla_YZ-\left<\nabla_X\nabla_YZ,v_0\right>v_0  -\left<\nabla_YZ,v_0\right>\nabla_Xv_0-\nabla_Y\nabla_XZ\\
&+\left<\nabla_Y\nabla_XZ,v_0\right>v_0  +\left<\nabla_XZ,v_0\right>\nabla_Yv_0-\nabla_{[X,Y]}Z +\left<\nabla_{[X,Y]}Z,v_0\right>v_0\\
&=(\Rm(X,Y)Z)^h +\left<Z,\nabla_Yv_0\right>\nabla_Xv_0 -\left<Z,\nabla_Xv_0\right>\nabla_Yv_0.
\end{split}
\]
\end{proof}

\smallskip

\section{Appendix II}

This section is devoted to the proof of Lemma \ref{relation} and \ref{relationxx}.

\begin{proof}[Proof of Lemma \ref{relation}]
By the definition of $\vec h_a$, we have $\pi_*(\vec h_a)=v_a$. Therefore, the first relation follows. The second relation follows from $\pi_*\vec\alpha_a=0$.

Let $\theta$ be the tautological 1-form defined by $\theta=p_adx_a$. Note that $\theta(\vec h_a)=h_a$ and $\omega=d\theta$. The third relation follows from
\[
\begin{split}
dh_b(\vec h_a)&=d\theta(\vec h_a,\vec h_b)\\
&=\vec h_a(\theta(\vec h_b))-\vec h_b(\theta(\vec h_a))-\theta([\vec h_a,\vec h_b])\\
&=2dh_b(\vec h_a)-(\Gamma_{ab}^c-\Gamma_{ba}^c)h_c.
\end{split}
\]

It is clear that $[\vec\alpha_a,\vec h_b]$ is vertical. The fourth relation follows from
\[
dh_c([\vec\alpha_a,\vec h_b])=\vec\alpha_a(dh_c(\vec h_b))=(\Gamma_{bc}^d-\Gamma_{cb}^d)dh_d(\vec\alpha_a)=\Gamma_{cb}^a-\Gamma_{bc}^a.
\]
The fifth and sixth relations follow from the fourth one and $\vec H=h_i\vec h_i$. The seventh follows from the third.

The eighth relation follows from the fifth and the sixth. Indeed,
\[
\begin{split}
&[\vec H,[\vec H,\vec\alpha_0]]=-\vec H(h_j\bJ_{jk})\vec\alpha_k-h_j\bJ_{jk}[\vec H,\vec\alpha_k]\\
&=-h_ldh_j(\vec h_l)\bJ_{jk}\vec\alpha_k-h_lh_j(v_l\bJ_{jk})\vec\alpha_k-h_j\bJ_{jk}\left(\vec h_k+h_l(\Gamma_{al}^k-\Gamma_{la}^k)\vec\alpha_a\right)\\
&=-h_lh_j\Gamma_{ls}^j\bJ_{sk}\vec\alpha_k+h_0h_k\vec\alpha_k -h_lh_j(\Gamma_{lj}^s\bJ_{sk}+\Gamma_{lk}^s\bJ_{js})\vec\alpha_k\\
&-h_j\bJ_{jk}\vec h_k-H\vec\alpha_0-h_jh_l\Gamma_{0l}^k\bJ_{jk}\vec\alpha_0-h_jh_l\bJ_{js}(\Gamma_{kl}^s-\Gamma_{lk}^s)\vec\alpha_k\\
&=h_0h_k\vec\alpha_k-h_j\bJ_{jk}\vec h_k-H\vec\alpha_0-h_jh_l\Gamma_{0l}^k\bJ_{jk}\vec\alpha_0-h_jh_l\bJ_{js}\Gamma_{kl}^s\vec\alpha_k.
\end{split}
\]

By the fifth relation, we have
\[
[\vec H,[\vec H,\vec\alpha_i]]=[\vec H,\vec h_i]+h_j(\Gamma_{kj}^i-\Gamma_{jk}^i)\vec h_k\quad \text{ (mod vertical)}
\]
when $i\neq 0$.

Since $\pi_*[\vec h_j,\vec h_k]=[v_j,v_k]$, the above equation becomes
\[
\begin{split}
&[\vec H,[\vec H,\vec\alpha_i]]\\
&=h_l(\Gamma_{li}^a-\Gamma_{il}^a)\vec h_a-h_a(\Gamma_{ik}^a-\Gamma_{ki}^a)\vec h_k+h_l(\Gamma_{kl}^i-\Gamma_{lk}^i)\vec h_k\quad \text{ (mod vertical)}\\
&=2h_l\Gamma_{li}^k\vec h_k+h_l\bJ_{li}\vec h_0-h_0\bJ_{ik}\vec h_k\quad \text{ (mod vertical)}.
\end{split}
\]

Finally, by the sixth relation, we have
\[
\begin{split}
&[\vec H,[\vec H,[\vec H,\vec\alpha_0]]]\\
&=-2\vec H(h_j\bJ_{jk})\vec h_k-h_j\bJ_{jk}[\vec H,[\vec H,\vec\alpha_k]] \text{(mod vertical)}\\
&=-2h_ldh_j(\vec h_l)\bJ_{jk}\vec h_k-2h_lh_j(v_l\bJ_{jk})\vec h_k\\
&-2h_lh_j\bJ_{jk}\Gamma_{lk}^i\vec h_i-2H\vec h_0-h_0\vec H\text{(mod vertical)}\\
&=-2h_ih_l\Gamma_{lj}^i\bJ_{jk}\vec h_k-2h_lh_j(v_l\bJ_{jk})\vec h_k\\
&-2h_lh_j\bJ_{jk}\Gamma_{lk}^i\vec h_i-2H\vec h_0+h_0\vec H\text{(mod vertical)}\\
&=-2h_lh_j(\bJ_{ik}\Gamma_{li}^j+\bJ_{ji}\Gamma_{li}^k+v_l\bJ_{jk})\vec h_k-2H\vec h_0+h_0\vec H\text{(mod vertical)}\\
\end{split}
\]

Since the manifold is Sasakian, we have
\[
\begin{split}
&[\vec H,[\vec H,[\vec H,\vec\alpha_0]]]=-2H\vec h_0+h_0\vec H\text{(mod vertical)}.
\end{split}
\]

\end{proof}

\begin{proof}[Proof of Lemma \ref{relationxx}]
Since $\pi_*\vec h_j=v_j$, $[\vec h_k,\vec h_i]$ is of the form
\[
[\vec h_k,\vec h_i]=(\Gamma_{ki}^a-\Gamma_{ik}^a)\vec h_a+b_{ki}^a\vec\alpha_a=\bJ_{ki}\vec h_0+b_{ki}^a\vec\alpha_a
\]
at $x$. By applying both sides by $dh_l$, we obtain
\[
\begin{split}
&-b_{ki}^0=dh_0[\vec h_k,\vec h_i]\\
&=\vec h_k(dh_0(\vec h_i))-\vec h_i(dh_0(\vec h_k))\\
&=\vec h_k[(\Gamma_{i0}^s-\Gamma_{0i}^s)h_s]-\vec h_i[(\Gamma_{k0}^s-\Gamma_{0k}^s)h_s]\\
&=h_0[\bJ_{ks}\bJ_{is}-\bJ_{is}\bJ_{ks}]+h_s[v_k(\Gamma_{i0}^s-\Gamma_{0i}^s)-v_i(\Gamma_{k0}^s-\Gamma_{0k}^s)]\\
&=h_s[v_k(\Gamma_{i0}^s-\Gamma_{0i}^s)-v_i(\Gamma_{k0}^s-\Gamma_{0k}^s)]=h_s[v_i\Gamma_{0k}^s-v_k\Gamma_{0i}^s]
\end{split}
\]
and
\[
\begin{split}
&\frac{1}{2}\bJ_{ki}\bJ_{ls}h_s-b_{ki}^l=dh_l[\vec h_k,\vec h_i]\\
&=\vec h_k(dh_l(\vec h_i))-\vec h_i(dh_l(\vec h_k))\\
&=\vec h_k[(\Gamma_{il}^a-\Gamma_{li}^a)h_a]-\vec h_i[(\Gamma_{kl}^a-\Gamma_{lk}^a)h_a]\\
&=-\frac{1}{2}\bJ_{il}\bJ_{ks}h_s+\frac{1}{2}\bJ_{kl}\bJ_{is}h_s+h_s[v_k(\Gamma_{il}^s-\Gamma_{li}^s)-v_i(\Gamma_{kl}^s-\Gamma_{lk}^s)]\\
\end{split}
\]
at $x$.

It also follows that
\[
\begin{split}
&h_kb_{ki}^0=h_kh_sv_k(\Gamma_{0i}^s),
\end{split}
\]
and
\[
\begin{split}
h_kb_{ki}^l&=-h_sh_k[v_k(\Gamma_{il}^s)-v_k(\Gamma_{li}^s)-v_i(\Gamma_{kl}^s)]
\end{split}
\]
at $x$.

Finally,
\[
[\vec H,\vec h_i]=h_k\bJ_{ki}\vec h_0-h_0\bJ_{ik}\vec h_k+h_k b_{ki}^a\vec\alpha_a.
\]
\end{proof}

\end{document}